\documentclass{amsart} 

\usepackage{tikz}
\usepackage[text={450pt,700pt},centering]{geometry}
\usepackage{color}
\usepackage{multicol}
\usepackage{empheq}
\usepackage{subfig}

\usepackage{amsmath, amsthm, amssymb, wasysym, verbatim, bbm, color, graphics}
\usepackage{url}
\usepackage{mathtools}
\usepackage[colorlinks=true, pdfstartview=FitV, linkcolor=blue, citecolor=blue, urlcolor=blue,pagebackref=false]{hyperref}\usepackage{romannum}
\usepackage{xfrac}
\usepackage{float}
\usepackage{amsmath,bm}
\usepackage{bbm}

\newcommand{\R}{\mathbb{R}}

\DeclareMathOperator{\E}{\mathcal{E}}

\newcommand\norm[1]{\left\lVert#1\right\rVert}

\renewcommand{\div}{\operatorname{div}}

\usepackage[capitalise]{cleveref}
\usepackage{autonum}

\newcommand{\epl}{\varepsilon}
\renewcommand{\d}{\mathrm{d}}
\newcommand{\dd}{\,\mathrm{d}}
\newcommand{\abs}[1]{\left\lvert#1\right\rvert}
\renewcommand{\E}{\mathbb{E}}

\crefname{assumption}{Assumption}{Assumptions}
\crefname{figure}{Figure}{Figures}

\usepackage[foot]{amsaddr}

\newtheorem{lemma}{Lemma}[section]
\newtheorem{proposition}[lemma]{Proposition}

\newtheorem{theorem}[lemma]{Theorem}
\newtheorem{remark}[lemma]{Remark}

\newtheorem{assumption}[lemma]{Assumptions}

\newtheorem*{maintheorem*}{Main Theorem}
\theoremstyle{definition}{}

\allowdisplaybreaks

\numberwithin{equation}{section}

\begin{document}

\pagenumbering{arabic}

\title[SGDCT for drift identification in multiscale diffusions]{Stochastic gradient descent in continuous time \\ for drift identification in multiscale diffusions}

\author{Max Hirsch}\address[MH]{Department of Mathematics, University of California Berkeley, Berkeley, CA 94720, USA}\email[]{ \texttt{mhirsch@berkeley.edu}}
\author{Andrea Zanoni}\address[AZ]{Institute for Computational and Mathematical Engineering, Stanford University, Stanford, CA 94305, USA}\email[]{\texttt{azanoni@stanford.edu}}
	
\begin{abstract}
We consider the setting of multiscale overdamped Langevin stochastic differential equations, and study the problem of learning the drift function of the homogenized dynamics from continuous-time observations of the multiscale system. We decompose the drift term in a truncated series of basis functions, and employ the stochastic gradient descent in continuous time to infer the coefficients of the expansion. Due to the incompatibility between the multiscale data and the homogenized model, the estimator alone is not able to reconstruct the exact drift. We therefore propose to filter the original trajectory through appropriate kernels and include filtered data in the stochastic differential equation for the estimator, which indeed solves the misspecification issue. Several numerical experiments highlight the accuracy of our approach. Moreover, we show theoretically in a simplified framework the asymptotic unbiasedness of our estimator in the limit of infinite data and when the multiscale parameter describing the fastest scale vanishes.
\end{abstract}

\subjclass{60H10, 60J60, 62F12, 62M05, 62M20}

\keywords{Filtered data, homogenization, multiscale diffusions, parameter estimation, stochastic gradient descent in continuous time.}

\maketitle

\section{Introduction}

Multiscale diffusion processes are widely employed in modeling natural phenomena. Some applications include oceanography \cite{CoP09,KPP15}, finance \cite{OSP10,AiJ14}, and molecular dynamics \cite{FrL12,LeS16}. It is often important to derive from data an effective model which describes the macroscopic behavior at the slowest scale \cite{AMZ05}. We focus here on multiscale stochastic differential equations (SDEs) of the Langevin type, which represents the motion of a particle in a multiscale potential, and for which a simplified surrogate model exists due to the theory of the homogenization \cite{PaS08}. Our goal is then to infer the drift coefficient of the effective equation given observations from the full dynamics. Since we are interested in learning the homogenized SDE from data, this problem is denoted as \emph{data-driven homogenization}.

It is well-known that data originating from the multiscale dynamics are not compatible with the coarse-grained model \cite{PaS07}, and this is a typical example of model misspecification. This problem has been studied in several works and tackled with different approaches \cite{PPS12,GaS17,GaS18}. In particular, it has been shown that the maximum likelihood estimator (MLE) is biased if data are not preprocessed. This issue can be solved by subsampling the data at an appropriate rate which lies in between the two characteristic time scales of the multiscale process \cite{PPS09,ABT10,ABJ13}. However, the accuracy of the estimation is strongly dependent on the subsampling rate, and the optimal rate is unknown in concrete applications. This has motivated the study of alternatives to subsampling. For example, in \cite{KPK13} a martingale property of the likelihood function is employed. Moreover, subsampling can be bypassed using a different methodology based on filtered data which has recently been introduced in \cite{AGP23,GaZ23}. Filtered data are constructed applying either an exponential kernel or a moving average, therefore smoothening the original trajectory. The filtered process is then used to modify the definition of the MLE, which results in a more robust estimator. In the case that the full path of the data is not available but only a sample of discrete-time observations is known, filtered data can be used in combination with martingale estimating functions based on the eigenvalues and eigenfunctions of the generator of the effective dynamics \cite{KeS99,APZ22}, for which homogenization results have been shown in \cite{Zan23}. In \cite{APZ22}, it is proved that, as long as the filter width is sufficiently large, the derived estimator is asymptotically unbiased independently of the regime at which the data are observed.

In this work we consider a different method for parameter estimation, i.e., the stochastic gradient descent in continuous time (SGDCT), which has been introduced and applied to financial problems in \cite{SiS17}. Differently from the stochastic gradient descent in discrete time, which has been thoroughly analyzed \cite{BMP90,BeT00,KuY03}, the continuous analogue applied to inference problems is rather new, and it has been studied in \cite{SiS17}, where the rate of convergence and a central limit theorem for the estimator are proved. SGDCT consists in solving an SDE for the estimator that is dependent on the model and the path of the observations, and therefore performing online updates. However, if the data are originated from a different dynamics, i.e., in this case the multiscale model insted of the homogenized dynamics, the estimator is not able to correctly retrieve the exact unknown due to the incompatibility between model and data, as already observed for the MLE. We therefore couple SGDCT together with the filtering methodology and build a novel estimator for learning the drift term of the homogenized equation from multiscale data. We remark that in \cite{PRZ24} a similar approach has been employed to estimate parameters in stochastic processes driven by colored noise.

The main contribution of our work is twofold. We show that SGDCT in combination with filtered data can be applied to infer drift coefficients in multiscale problems, and we prove theoretically in a particular setting that the resulting estimator is asymptotically unbiased in the limit of infinite data and vanishing fast-scale parameter. Moreover, we apply our estimators to the general framework of multiscale diffusion processes with nonseparable confining potential, where the amplitude of the fast-scale oscillations are dependent on the slow-scale component, and which lead to multiplicative noise in the homogenized equation \cite{DKP16}.

\vspace{2ex}
\textit{Outline.} The rest of the paper is organized as follows. In \cref{sec:setting} we introduce the problem and the multiscale model under consideration. In \cref{sec:method} we present our methodology based on SGDCT and filtered data, for which we study the convergence properties in a simplified setting in \cref{sec:analysis}. Finally, in \cref{sec:experiments} we show the potentiality of our estimator through several numerical experiments, and in \cref{sec:conclusion} we draw the conclusions.

\section{Problem setting} \label{sec:setting}

In this section we introduce the multiscale diffusion process which we study in this work. We focus on multiscale Langevin dynamics, which describes the motion of a particle in a multiscale landscape. Let us consider the following SDE in the time horizon $[0,T]$
\begin{equation} \label{eq:SDE_multiscale}
\d X_t^\epl = - \nabla \mathcal V^\epl(X_t^\epl; \alpha) \dd t + \sqrt{2\sigma} \dd W_t,
\end{equation}
where $X_t \in \R^d$, $\mathcal V^\epl \colon \R^d \to \R$ defines a multiscale confining potential depending on an unknown parameter $\alpha \in \R^m$ and on a parameter $\epl > 0$ that describes the fastest scale, $\sigma > 0$ is a constant diffusion coefficient, and $W_t \in \R^d$ is a standard multidimensional Brownian motion. We decompose the confining potential in its large-scale part and its zero-mean fluctuations as
\begin{equation}
\mathcal V^\epl(x; \alpha) = V(x; \alpha) + p \left( x, \frac{x}\epl \right),
\end{equation}
where $V \colon \R^d \to \R$ and $p \colon \R^{2d} \to \R$ are the slow-scale and fast-scale potentials, respectively, and we assume that the parameter $\alpha$ appears only in the first component. Moreover, the fast-scale potential is periodic in its second component with period the hypercube $\mathbb L = [0,L]^d$ with $L>0$, and it satisfies the centering condition
\begin{equation}
\int_{\mathbb L} p(x,y) \dd y = 0,
\end{equation}
which, defining $\mathcal V(x,y;\alpha) = V(x;\alpha) + p(x,y)$, implies
\begin{equation}
V(x;\alpha) = \frac1{L^d} \int_{\mathbb L} \mathcal V(x,y;\alpha) \dd y.
\end{equation}
The theory of homogenization (see, e.g., \cite[Chapter 9]{BLP78} and \cite{DKP16}) gives the existence of a homogenized SDE, which only captures the slow-scale components of the solution $X_t^\epl$ of the multiscale equation \eqref{eq:SDE_multiscale}, of the form
\begin{equation} \label{eq:SDE_homogenized}
\d X_t = - b(X_t; \alpha) \dd t + \sqrt{2 \Sigma(X_t)} \d W_t,
\end{equation}
such that the solution $X_t$ is the limit in law of $X_t^\epl$ as random variables in $\mathcal C^0([0,T]; \R^d)$. In particular, in equation \eqref{eq:SDE_homogenized} the homogenized drift and diffusion terms are given by $\Sigma(x) = \sigma \mathcal K(x)$ and
\begin{equation} \label{eq:drift_hom}
b(x; \alpha) = \mathcal K(x) \nabla V(x; \alpha) - \sigma \mathcal K(x) \nabla \log \left( \int_{\mathbb L} e^{-\frac1\sigma p(x,y)} \dd y \right) - \sigma \div(\mathcal K(x)),
\end{equation}
where $\mathcal K \colon \R^d \to \R^{d \times d}$ is defined as
\begin{equation}
\mathcal K(x) = \frac{\int_{\mathbb L} \left( I + \nabla_y \Phi(x,y) \right) e^{-\frac1\sigma p(x,y)} \dd y}{\int_{\mathbb L} e^{-\frac1\sigma p(x,y)} \dd y},
\end{equation}
and where, for fixed $x \in \R^d$, the function $\Phi \colon \R^{2d} \to \R^d$ is the unique solution to the cell problem in $\mathbb L$ 
\begin{equation} \label{eq:cell_problem}
\nabla_y \cdot \left( e^{-\frac1\sigma p(x,y)} \left( I + \nabla_y \Phi(x,y) \right) \right) = 0,
\end{equation}
endowed with periodic boundary conditions, and which satisfies the constraint 
\begin{equation}
\int_{\mathbb L} \Phi(x,y) e^{-\frac1\sigma p(x,y)} \dd y = 0. 
\end{equation}
We remark that similar computations have also been presented in \cite{DDP23} for an arbitrary number of scales, and in \cite{GoP18} for the more general setting of multiscale interacting particles.

\begin{remark}
In the one-dimensional case, equation \eqref{eq:cell_problem} can be solved analytically, and therefore we obtain the explicit formula for the function $\mathcal K$, which reads
\begin{equation} \label{eq:K1D}
\mathcal K(x) = \frac{L^2}{ \left( \int_0^L e^{-\frac1\sigma p(x,y)} \dd y \right) \left( \int_0^L e^{\frac1\sigma p(x,y)} \dd y \right)}.
\end{equation}
\end{remark}

Let us now introduce the framework that we will consider in the rest of the work, and in particular the assumptions on the confining potentials.

\begin{assumption} \label{as:potentials}
The slow-scale and fast-scale potentials $V$ and $p$ satisfy
\begin{enumerate}
\item \label{it:1} $V(\cdot; \alpha) \in \mathcal C^\infty(\R^d)$ is polynomially bounded, and there exist $\mathfrak a, \mathfrak b > 0$ such that
\begin{equation}
- \nabla V(x; \alpha) \cdot x \le \mathfrak a - \mathfrak b \norm{x}^2,
\end{equation} 
\item \label{it:2} $p \in \mathcal C^\infty(\R^{2d})$ is periodic in its second variable with period $\mathbb L$, and either there exists $R > 0$ such that $p(x,y) = 0$ if $\norm{x} > R$ or $p$ is independent of the first variable, i.e., by abuse of notation, $p(x,y) = p(y)$.
\end{enumerate}
\end{assumption}

We notice that the dissipativity condition in \cref{as:potentials}(\ref{it:1}) is important to ensure the ergodicity of the diffusion process. Moreover, \cref{as:potentials}(\ref{it:2}) is necessary to prevent the oscillations from growing when $\norm{x} \to \infty$. This condition is essential to apply the a priori estimates on the moments from \cite{ABM96} and showing the homogenization result \cite{DDP23}.

We are then interested in inferring the drift function $b$ of the homogenized equation \eqref{eq:SDE_homogenized} given a continuous path $(X_t^\epl)_{t \in [0,T]}$ from the corresponding multiscale SDE \eqref{eq:SDE_multiscale}. We remark that we only assume to know the form of the slow-scale potential $V$ (but not the parameter $\alpha$), while all the other quantities, i.e., the fast-scale potential $p$, the diffusion coefficient $\sigma$, and therefore the function $\mathcal K$, are unknown for us. So if, for example, $V$ was known to be a fourth degree polynomial with coefficients given by $\alpha \in \R^4$, we would not assume to know $\alpha$. The problem of inferring the drift of the homogenized equation, rather than the full multiscale dynamics, is important in particular for predictive purposes \cite[Remark 1.1]{GaZ23}, since numerical integration for multiscale SDEs is prohibitively expensive. In fact, the discretization step needs to be chosen sufficiently small with respect to $\epl$. This restriction is not necessary in homogenized SDEs where the time step is only based on the desired accuracy. In the next section we address the inference problem employing the SGDCT. However, since the data and model are incompatible, we also propose to include filtered data in the definition of the SGDCT in order to solve the issue of model misspecification. 

\section{Methodology} \label{sec:method}

We recall that our goal is to learn the drift function $b$ in equation \eqref{eq:SDE_homogenized} from multiscale data $(X_t^\epl)_{t \in [0,T]}$. From now on, we will focus on one-dimensional diffusion processes ($d=1$), but we believe that this method can be extended to the multidimensional case in a similar way as it is done for the MLE in \cite[Remark 2.5]{AGP23} and \cite[Section 3.3]{GaZ23}. We propose to follow a semiparametric approach. In particular, we first consider a truncated basis $\{ u_n \}_{n=1}^N$, $u_n \colon \R \to \R$, of an appropriate functional space, and approximate the drift term as
\begin{equation} \label{eq:expansion}
b(x; \alpha) \simeq \widetilde b(x; A) = \sum_{n=1}^N a_n u_n(x) = A^\top U(x),
\end{equation}
for some coefficients $a_n \in \R$, and where
\begin{equation}
U(x) = \begin{bmatrix} u_1(x) & u_2(x) & \cdots & u_N(x) \end{bmatrix}^\top \in \R^{N} \qquad \text{and} \qquad A = \begin{bmatrix} a_1 & a_2 & \cdots & a_N \end{bmatrix}^\top \in \R^N.
\end{equation}
The basis can be chosen, e.g., as the first $N$ monomials, i.e., $u_n(x) = x^{n-1}$, if we assume a Taylor expansion, or can be extracted from a family of orthogonal polynomials, or can be informed by the knowledge of the form of the slow-scale potential $V$. In particular, the form of the potential $V$ is the only guide in the choice of the basis functions. Otherwise, if no information is given, and assuming that the potential is regular enough, choosing monomials is the best option, which is justified by the Taylor expansion. Another interesting possibility would be to select polynomials that are orthonormal with respect to the invariant measure of the dynamics, but this would first require density estimation for the invariant measure. We plan to return to this problem in future work. Then, the problem moves to estimating the best coefficients $\{ a_n \}_{n=1}^N$ in \eqref{eq:expansion}, using estimators which are suitable for the standard parametric framework. We remark that the MLE cannot be computed explicitly as it depends on the diffusion term in \eqref{eq:SDE_homogenized}, which is also unknown. In this work, we therefore employ a different technique, i.e., the SGDCT estimator which has been used for parameter estimation in \cite{SiS17}, and does not rely on the diffusion coefficient.

\subsection{Application of SGDCT}\label{sec:application-of-sgdct}

In this section we introduce the methodology which we employ to estimate the drift coefficients in the approximation \eqref{eq:expansion}. In particular, we employ the SGDCT, which consists in solving an SDE for the unknown parameters. SGDCT follows a noisy descent direction along the path of the observations and therefore performs online updates. Following the notation of \cite{SiS17}, consider the natural objective function
\begin{equation}
\bar{g}(A) = \frac12 \E\left[\abs{b(X;\alpha) - \widetilde{b}(X; A)}^2\right] = \frac12 \E \left[ \abs{b(X;\alpha) - A^\top U(X)}^2 \right],
\end{equation}
and use stochastic gradient descent to approximate its minimum value with some learning rate $\eta_t$:
\begin{equation}
\d A_t = - \eta_t \nabla \bar{g}(A_t) \dd t = - \eta_t \left( A_t^\top U(X_t) - b(X_t;\alpha) \right) U(X_t) \dd t.
\end{equation}
We notice that, since $\alpha$ is unknown, the right-hand side cannot be computed explicitly, but it can be approximated using equation \eqref{eq:SDE_homogenized}
\begin{equation}
- b(X_t;\alpha) \dd t \simeq \dd X_t,
\end{equation}
where the noise is neglected. Therefore, we obtain
\begin{equation}
\d A_t = - \eta_t (U(X_t) \otimes U(X_t)) A_t \dd t - \eta_t U(X_t) \dd X_t,
\end{equation}
and an analogous estimator is obtained by replacing $X_t$ with the available multiscale data. Let us denote by $\widetilde A^\epl_T$ the SGDCT estimator of $A$ in \eqref{eq:expansion} given continuous-time observations $(X_t^\epl)_{t \in [0,T]}$. Then, $\widetilde A^\epl_T$ is the solution at the final time $T$ of the following SDE
\begin{equation} \label{eq:SDE_estimator}
\d \widetilde A^\epl_t = - \eta_t (U(X_t^\epl) \otimes U(X_t^\epl)) \widetilde A^\epl_t \dd t - \eta_t U(X_t^\epl) \dd X_t^\epl,
\end{equation}
where $\otimes$ denotes the outer product and $\eta_t$ is the learning rate which takes the form
\begin{equation}
\eta_t = \frac{\gamma}{\beta + t},
\end{equation}
for some $\gamma, \beta > 0$. As shown in \cite{SiS17,SiS20}, this method is suitable to approximate the exact coefficients if the data employed in the SDE \eqref{eq:SDE_estimator} are generated from the same equation used in the derivation of the SGDCT estimator. To be more precise, this technique would succeed in giving a reliable estimation of the coefficient $A$ if the data originated from the homogenized equation \eqref{eq:SDE_homogenized} where the drift term $b$ is replaced by its approximation $\widetilde b$. We notice that in this case the estimator $A_t$ defined above is independent of the parameter $\epl$ since it is based on the homogenized data $X_t$. In our setting, however, the drift term is not known exactly, but we use its approximation in terms of basis functions, and the data are generated from the multiscale equation \eqref{eq:SDE_multiscale}. We remark that, as long as the basis functions are chosen appropriately, the fact that we are looking for an approximation of the drift function does not strongly affect the inference procedure, as we will see in the numerical experiments. Nevertheless, the best possible approximation is clearly dependent on the choice of basis functions. On the other hand, even if the drift can be exactly expressed in terms of a finite number of basis functions, in the next section we observe that the SGDCT estimator $\widetilde A^\epl_T$ alone is not able to correctly infer the drift coefficient $A$ due to the incompatibility between multiscale data and homogenized model.

\subsection{Failure of SGDCT with multiscale data} \label{sec:failure}

Let us now show that if we do not preprocess the multiscale data, then the SGDCT cannot infer the drift function of the homogenized equation. This has already been studied for the MLE estimator in \cite{PaS07}, where the authors show that the estimator approximates the drift coefficient of the multiscale equation \cite[Theorem 3.4]{PaS07}. We consider here the same setting, where there is a clear separation between the slowest and the fastest scale. In particular, let the process $X_t^\epl$ be one-dimensional and consider the slow-scale potential
\begin{equation}
V(x) = \alpha v(x),
\end{equation}
where $\alpha > 0$ is the unknown parameter, and the fast-scale potential $p(x,y) = p(y)$ depends only on the fastest scale. In this case the function $\mathcal K$ in \eqref{eq:K1D} turns out to be constant and therefore the multiscale and homogenized equations \eqref{eq:SDE_multiscale} and \eqref{eq:SDE_homogenized} read
\begin{align}
\d X_t^\epl &= - \alpha v'(X_t^\epl) \dd t - \frac1\epl p' \left( \frac{X_t^\epl}\epl \right) \dd t + \sqrt{2\sigma} \dd W_t,  \label{eq:SDE_multiscale_1D} \\
\d X_t &= - A v'(X_t) \dd t + \sqrt{2\Sigma} \dd W_t,  \label{eq:SDE_homogenized_1D} \\
\end{align}
where $A = \mathcal K \alpha > 0$ and $\Sigma = \mathcal K \sigma > 0$. It follows that the only basis function that we need is $u_1(x) = v'(x)$ and thus the SDE for the estimator $\widetilde A^\epl_T$ of $A$ simplifies to
\begin{equation} \label{eq:SDE_estimator_1D}
\d \widetilde A^\epl_t = - \eta_t v'(X_t^\epl)^2 \widetilde A^\epl_t \dd t - \eta_t v'(X_t^\epl) \dd X_t^\epl.
\end{equation}
Since it will be useful in the following result, we remark here that in \cite[Propositions 5.1 and 5.2]{PaS07} it is proved that the processes $X_t^\epl$ and $X_t$ are geometrically ergodic with unique invariant measures. Then, in the next result we show that, even in this simple setting, the estimator is not able to approximate the drift coefficient $A$ of the homogenized equation \eqref{eq:SDE_homogenized_1D}, and, in contrast, it converges to the drift coefficient $\alpha$ of the multiscale equation \eqref{eq:SDE_multiscale_1D} in the limit of infinite time and when the multiscale parameter vanishes. 

\begin{proposition}
Let $\widetilde A^\epl_T$ be defined by the SDE \eqref{eq:SDE_estimator_1D}. Under \cref{as:potentials} it holds
\begin{equation}
\lim_{\epl \to 0} \lim_{T \to \infty} \widetilde A^\epl_T = \alpha, \qquad a.s.
\end{equation}
\end{proposition}
\begin{proof}
Notice that solving the SDE \eqref{eq:SDE_estimator_1D} is equivalent to applying the SGDCT to the multiscale equation \eqref{eq:SDE_multiscale_1D} so that we learn the unknown coefficient $A$ in the function $\widetilde{b}(x; A) := -Av'(x)$ which is used to approximate the target function
\begin{equation}
b(x; \alpha) := - \frac{ \partial \mathcal V^\epl}{\partial x} (x; \alpha) = - \alpha v'(x) - \frac1\epl p' \left( \frac{x}\epl \right).
\end{equation}
Similar to the intuitive derivation of SGDCT at the beginning of Section \ref{sec:application-of-sgdct} and now employing the notation in \cite{SiS20}, let us now define the objective function $g$ for SGDCT as
\begin{equation}
g(x; A) = \frac12\abs{\widetilde{b}(x; A) - b(x;\alpha)}^2 = \frac12 \abs{- A v'(x) + \alpha v'(x) + \frac1\epl p' \left( \frac{x}\epl \right)}^2.
\end{equation}
By \cite[Theorem 3]{SiS20} it follows that 
\begin{equation} \label{eq:fail_limit_dg}
\lim_{T \to \infty} \abs{\bar g'(\widetilde A^\epl_T)} = 0, \qquad a.s., 
\end{equation}
where $\bar g(a) = \E [g(X^\epl; a)]$ and the expectation is computed with respect to the invariant measure of the multiscale process $X_t^\epl$. We then have
\begin{equation}
\bar g'(\widetilde A^\epl_T) = \E \left[ \frac{\partial g}{\partial a} (X^\epl; \widetilde A^\epl_T) \right]  = \widetilde A^\epl_T \E[v'(X^\epl)^2] - \alpha \E[v'(X^\epl)^2] - \E \left[ \frac1\epl p' \left( \frac{X^\epl}\epl \right) v'(X^\epl) \right],
\end{equation}
which due to \eqref{eq:fail_limit_dg} implies
\begin{equation} \label{eq:fail_limit_T}
\lim_{T \to \infty} \widetilde A^\epl_T = \alpha + \frac{\E \left[ \frac1\epl p' \left( \frac{X^\epl}\epl \right) v'(X^\epl) \right]}{\E[v'(X^\epl)^2]}.
\end{equation}
Finally, from the proof of \cite[Theorem 3.4]{PaS07} we know that
\begin{equation}
\lim_{\epl \to 0} \frac{\E \left[ \frac1\epl p' \left( \frac{X^\epl}\epl \right) v'(X^\epl) \right]}{\E[v'(X^\epl)^2]} = 0,
\end{equation}
which together with equation \eqref{eq:fail_limit_T} yields the desired result.
\end{proof}

Therefore, in the next section we show how to modify the SGDCT in order to accurately fit the homogenized equation.

\subsection{Coupling SGDCT with filtered data}

In the previous section we showed that the SGDCT fails in inferring parameters in homogenized equations given observations originating from the multiscale dynamics. Inspired by the works \cite{AGP23,GaZ23}, where filtered data are used to correct the faulty behavior of MLE estimators, we now propose to employ the same type of filtered data to modify the SGDCT and consequently obtain unbiased estimations. Let $\delta > 0$ be a parameter describing the filter width, and consider the following smoothed versions of the original multiscale data
\begin{equation} \label{eq:filtered_data}
Z^{\delta, \epl}_{\mathrm{exp}, t} =  \frac1\delta \int_0^t e^{- \frac{t-s}{\delta}} X_s^\epl \dd s \qquad \text{and} \qquad Z^{\delta, \epl}_{\mathrm{ma}, t} = \frac1\delta \int_{t-\delta}^t X_s^\epl \dd s,
\end{equation}
where $Z^{\delta, \epl}_{\mathrm{exp}, t}$ is obtained applying a filtering kernel of the exponential family and $Z^{\delta, \epl}_{\mathrm{ma}, t}$ is a moving average. To simplify the notation, from now on we will only write $Z^\epl_t$ and we will refer to both as the filtered data in equation \eqref{eq:filtered_data}. However, if it is necessary to specify the type, then we will use the full notation. For completeness, let us also introduce the homogenized counterpart of the filtered data
\begin{equation} \label{eq:filtered_data_hom}
Z^\delta_{\mathrm{exp}, t} =  \frac1\delta \int_0^t e^{- \frac{t-s}{\delta}} X_s \dd s \qquad \text{and} \qquad Z^\delta_{\mathrm{ma}, t} = \frac1\delta \int_{t-\delta}^t X_s \dd s,
\end{equation}
which will be employed later and only denoted by $Z_t$. We are now ready to include filtered data in the definition of the SGDCT estimator. In particular, inspired by the previous works \cite{AGP23,GaZ23,APZ22}, we replace only one instance of the original trajectory $X_t^\epl$ with the filtered data $Z_t^\epl$ both in the drift and diffusion terms of the SDE \eqref{eq:SDE_estimator}. Our novel estimator, which we denote as $\widehat A^\epl_T$, is then the solution at the final time $T$ of the modified SDE
\begin{equation} \label{eq:SDE_estimator_filter}
\d \widehat A^\epl_t = - \eta_t (U(Z_t^\epl) \otimes U(X_t^\epl)) \widetilde A^\epl_t \dd t - \eta_t U(Z_t^\epl) \dd X_t^\epl,
\end{equation}
We notice that, as highlighted in \cite[Remark 3.7]{AGP23}, due to the regularity of the filtered process $Z_t^\epl$, it is fundamental to keep the differential of the multiscale process $\d X_t^\epl$ in the definition of the estimator. Then, in order to obtain the asymptotic unbiasedness, by symmetry, it follows that the original process must also appear in the drift term of the SDE \eqref{eq:SDE_estimator_filter}. This intuition will become clearer in the following convergence analysis.

\begin{remark}
The computation of the estimator $\widehat A^\epl_T$ from equation \eqref{eq:SDE_estimator_filter} requires neither the knowledge of the diffusion coefficient $\sigma$ of the multiscale dynamics \eqref{eq:SDE_multiscale} nor the diffusion term $\Sigma$ of the homogenized equation \eqref{eq:SDE_homogenized}. However, their values affect the variance of the resulting estimator, and this can be observed, e.g., in the central limit theorem proved in \cite[Theorem 2.8]{SiS20}.
\end{remark}

Finally, from the expansion \eqref{eq:expansion} we can construct the approximated drift function of the homogenized equation which is given by
\begin{equation}
\widetilde b(x; \widehat A^\epl_T) = \sum_{n=1}^N (\widehat A^\epl_T)_n u_n(x).
\end{equation}
In the numerical experiments in \cref{sec:experiments} we will show that the estimated drift function $\widetilde b$ is able to provide a reliable approximation of the exact drift function $b$. Moreover, in the next section we focus on a particular framework, and develop a rigorous theoretical analysis which allows us to prove the asymptotic unbiasedness of the estimator $\widehat A^\epl_T$ in that context.

\section{Convergence analysis} \label{sec:analysis}

In this section we consider a simplified setting and prove theoretically the accuracy of our estimator in the limit of infinite data ($T \to \infty$) and when the multiscale parameter vanishes ($\epl \to 0$). Let us consider the same framework as in \cref{sec:failure}, which we recall here for clarity. We aim to infer the drift coefficient $A > 0$ in the homogenized equation
\begin{equation} \label{eq:SDE_homogenized_proof}
\d X_t = - A v'(X_t) \dd t + \sqrt{2\Sigma} \dd W_t,
\end{equation}
given continuous-time observations $(X_t^\epl)_{t \in [0,T]}$ from the corresponding multiscale dynamics
\begin{equation} \label{eq:SDE_multiscale_proof} 
\d X_t^\epl = - \alpha v'(X_t^\epl) \dd t - \frac1\epl p' \left( \frac{X_t^\epl}\epl \right) \dd t + \sqrt{2\sigma} \dd W_t.
\end{equation}
We employ the SGDCT estimator $\widehat A^\epl_T$ with filtered data which is the solution at the final time of the SDE
\begin{equation} \label{eq:SDE_estimator_proof}
\d \widehat A^\epl_t = - \eta_t v'(Z_t^\epl) v'(X_t^\epl) \widehat A^\epl_t \dd t - \eta_t v'(Z_t^\epl) \dd X_t^\epl.
\end{equation}
Substituting the expression for $\d X_t^\epl$ in \eqref{eq:SDE_multiscale_proof} into that of $\d \widehat A^\epl_t$ in \eqref{eq:SDE_estimator_proof} and using $\d (\widehat A^\epl_t - \alpha) = \d \widehat A^\epl_t$ gives
\begin{equation}
\d (\widehat A^\epl_t - \alpha) = - \eta_t v'(Z_t^\epl) v'(X_t^\epl) (\widehat A^\epl_t - \alpha) \dd t + \frac{\eta_t}\epl v'(Z_t^\epl) p' \left( \frac{X_t^\epl}\epl \right) \dd t - \sqrt{2\sigma} \eta_t v'(Z_t^\epl) \dd W_t,
\end{equation}
which is a linear SDE for the term $\widehat A^\epl_t - \alpha$. Therefore, letting $A_0 > 0$ be the initial value for the estimator, we can write the solution explicitly and obtain the analytical form of the estimator
\begin{equation} \label{eq:estimator_proof}
\begin{aligned}
\widehat A^\epl_T  &= \alpha + (A_0 - \alpha) e^{- \int_0^T \eta_s v'(Z_s^\epl) v'(X_s^\epl) \dd s} - \sqrt{2\sigma} \int_0^T \eta_t e^{- \int_t^T \eta_s v'(Z_s^\epl) v'(X_s^\epl) \dd s} v'(Z_t^\epl) \dd W_t \\
&\quad + \frac1\epl \int_0^T \eta_t e^{- \int_t^T \eta_s v'(Z_s^\epl) v'(X_s^\epl) \dd s} v'(Z_t^\epl) p' \left( \frac{X_t^\epl}\epl \right) \dd t. 
\end{aligned}
\end{equation}
Our goal is then to show the asymptotic unbiasedness of $\widehat A^\epl_T$, i.e., its convergence to the exact coefficient $A$ when $T$ and $\epl$ tend to infinity and zero, respectively. We remark that the following analysis relies on the additional assumption that the state space is compact, i.e., we embed the multiscale diffusion process $X_t^\epl$ in the one-dimensional torus, considering the wrapping of an Euclidean diffusion process, as detailed in \cite{GSM19}. The setting in which our analysis is performed is summarized in \cref{as:proof} below for the clarity of the exposition.

\begin{assumption} \label{as:proof}
The following conditions hold.
\begin{enumerate}
\item \label{it:P1} The dimensions of the processes and of the unknown parameters are $d=1$ and $N=1$, respectively.
\item \label{it:P2} The slow-scale potential is $V(x; \alpha) = \alpha v(x)$, where $\alpha > 0$ and $v$ is $\mathcal C^\infty(\R)$ and periodic.
\item \label{it:P3} The filtered data $Z_t^\epl$ and $Z_t$ are obtained employing the exponential kernel as in equations \eqref{eq:filtered_data} and \eqref{eq:filtered_data_hom}.
\item \label{it:P4} The fast-scale potential $p  \in \mathcal C^\infty(\R)$ is only dependent on the fast component and is periodic with period $L > 0$.
\item \label{it:P5} $\E[v'(X) v'(Z)] > 0$, where the expectation is computed with respect to the invariant measure of the joint homogenized process $(X_t, Z_t)$, which is introduced at the beginning of \cref{sec:preliminary}.
\item \label{it:P6} The homogenized and multiscale diffusion processes are embedded in the one-dimensional torus and therefore the state space is compact. 
\item \label{it:P7} The learning rate has the form $\eta_t = \gamma/(\beta + t)$ for some $\gamma, \beta > 0$, and $\gamma$ is chosen sufficiently large such that $\gamma \E[v'(X) v'(Z)] > 1$.
\end{enumerate}
\end{assumption}

In this context we can now present the asymptotic unbiasedness of our estimator, which is the main result of this section and whose proof is outlined below.

\begin{theorem} \label{thm:main}
Let $\widehat A^\epl_T$ be defined in equation \eqref{eq:estimator_proof}. Under \cref{as:proof} and if the filter width $\delta$ is independent of $\epl$ or $\delta = \epl^\xi$ with $\xi \in (0,2)$, i.e., if $\delta$ is sufficiently large, it holds
\begin{equation}
\lim_{\epl \to 0} \lim_{T \to \infty} \widehat A^\epl_T = A, \qquad \text{in } L^2.
\end{equation}
\end{theorem}

\begin{remark}
 Even if \cref{thm:main} is proved in the particular setting described by \cref{as:proof}, the following analysis highlights the importance of filtered data in correcting the path of the SGDCT and therefore giving an unbiased estimator. We also remark that, as showcased by the numerical experiments, our methodology works even in more complex scenarios. Hence, we believe that the theoretical results could be extended to more general classes of problems. However, we will not study here this extension and we leave it for future work. Nevertheless, even if more general results are out of the scope of this work, we present the main issues that we would encounter and ideas on how we might tackle them. First, in the multidimensional case we could not compute the analytical solution in \eqref{eq:estimator_proof}, and we would need to use arguments similar to the ones in the proof of \cite[Theorem 1]{SiS20}, such as the comparison theorem \cite[Theorem 1.1]{IkW77}. Then, performing the analysis in the whole space, rather than the compact torus, would require the study of the regularity and boundedness of the solution of Poisson problems for hypoelliptic operators in unbounded domains, by extending the series of works \cite{PaV01,PaV03,PaV05}. Replacing exponentially filtered data with averaged data would lead to the analysis of the generator of stochastic delay differential equations (see \cite{GaZ23}). Finally, considering fast-scale potentials depending on both scales would imply the need of the expansion \eqref{eq:expansion}, and therefore the study of the error committed in approximating the drift with the truncated series expansion.
\end{remark}

In the next two sections we first introduce technical results related to the homogenization and the generator of the multiscale diffusion process, and then present the proof of the asymptotic unbiasedness of the SGDCT estimator $\widehat A^\epl_t$ with filtered data.

\subsection{Preliminary results} \label{sec:preliminary}

Let us recall some useful properties of the joint multiscale and homogenized processes $(X_t^\epl, Z_t^\epl)$ and $(X_t, Z_t)$, that has been studied in \cite{AGP23}. In particular, the joint process $(X_t^\epl, Z_t^\epl)$ is hypoelliptic (\cite[proof of Lemma 3.2]{AGP23}), ergodic, and admits a unique invariant measure $\mu^\epl$ (\cite[Lemma 3.3]{AGP23}). Moreover, $\mu^\epl$ converges in law to the unique invariant measure $\mu$ of the homogenized joint process $(X_t, Z_t)$ by \cite[Lemma 3.9]{AGP23}. Therefore, it follows that there exists $\epsilon_0 > 0$ such that for all $\epl < \epsilon_0$ \cref{as:proof}(\ref{it:P5},\ref{it:P7}) hold also for the invariant measure of the multiscale process, namely
\begin{equation} \label{eq:as_epsilon}
\E[v'(X^\epl) v'(Z^\epl)] > 0 \qquad \text{and} \qquad \gamma \E[v'(X^\epl) v'(Z^\epl)] > 1.
\end{equation}
The next lemma is fundamental for the proof of \cref{thm:main}. 

\begin{lemma}  \label{lem:magic_formula}
Under \cref{as:proof}, if the filter width $\delta$ is independent of $\epl$ or $\delta = \epl^\xi$ with $\xi \in (0,2)$, i.e., if $\delta$ is sufficiently large, it holds
\begin{equation}
\lim_{\epl \to 0} \frac{\E \left[ v'(Z^\epl) p' \left( \frac{X^\epl}\epl \right) \right]}{\epl \E \left[ v'(Z^\epl) v'(X^\epl) \right]} = A - \alpha,
\end{equation}
where the expectations are computed with respect to the invariant measure of the joint process $(X_t^\epl, Z_t^\epl)$, and $A$ and $\alpha$ are the drift coefficients of the homogenized and multiscale SDEs \eqref{eq:SDE_homogenized_proof} and \eqref{eq:SDE_multiscale_proof}, respectively.
\end{lemma}
\begin{proof}
The desired result follows from the proofs of \cite[Theorem 3.12]{AGP23} and \cite[Theorem 3.18]{AGP23}.
\end{proof}

We now focus on the definition of the estimator \eqref{eq:estimator_proof}. The next result is a decomposition of the exponent of the exponential term, and is based on the Poisson problem for the generator of the joint process $(X_t^\epl, Z_t^\epl)$.

\begin{lemma} \label{lem:decomposition_poisson}
Under \cref{as:proof}, there exists a function $\psi = \psi(x,z)$, which is bounded along with its derivatives, such that the following decomposition holds true for all $s,t \in [0,T]$
\begin{equation}
\begin{aligned}
\int_s^t \eta_r v'(Z_r^\epl) v'(X_r^\epl) \dd r &= \eta_t \psi(X_t^\epl, Z_t^\epl) - \eta_s \psi(X_s^\epl, Z_s^\epl) + \int_s^t \eta_r \dd r \E[v'(Z^\epl) v'(X^\epl)] \\
&\quad - \int_s^t \eta_r' \psi(X_r^\epl, Z_r^\epl) \dd r - \sqrt{2\sigma} \int_s^t \eta_r \frac{\partial}{\partial x} \psi(X_r^\epl, Z_r^\epl) \dd W_r.
\end{aligned}
\end{equation}
\end{lemma}
\begin{proof}
Let us consider the system of SDEs for the joint process 
\begin{equation}
\begin{aligned}
\d X_t^\epl &= - \alpha v'(X_t^\epl) \dd t - \frac1\epl p' \left( \frac{X_t^\epl}\epl \right) \dd t + \sqrt{2\sigma} \dd W_t, \\
\d Z_t^\epl &= - \frac1\delta (Z_t^\epl - X_t^\epl) \dd t,
\end{aligned}
\end{equation}
which is obtained differentiating the first equation in \eqref{eq:filtered_data}, and its generator
\begin{equation}
\mathcal L^\epl = - \alpha v'(x) \frac{\partial}{\partial x} - \frac1\epl p' \left( \frac{x}\epl \right) \frac{\partial}{\partial x} - \frac1\delta (z-x)  \frac{\partial}{\partial z} + \sigma  \frac{\partial^2}{\partial x^2},
\end{equation}
which is hypoelliptic due to \cite[Lemma 3.2]{AGP23}.
Then, take the function 
\begin{equation}
\phi^\epl(x,z) = v'(z) v'(x) - \E [v'(Z^\epl) v'(X^\epl)],
\end{equation}
where the expectation is computed with respect to the invariant measure $\mu^\epl$, and consider the Poisson problem on the two-dimensional torus $\mathbb T^2$ with a centering condition
\begin{equation} \label{eq:poisson_problem}
\begin{aligned}
\mathcal L^\epl \psi^\epl &= \phi^\epl, \\
\int \psi^\epl \dd \mu^\epl &= 0,
\end{aligned}
\end{equation}
that has a unique solution, which is bounded along with its derivatives, by \cite[Theorem 4.1]{MST10}. Applying Itô's lemma to the function $h^\epl(r,x,z) = \eta_r \psi^\epl(x,z)$ we have
\begin{equation}
\eta_t \psi(X_t^\epl, Z_t^\epl) - \eta_s \psi(X_s^\epl, Z_s^\epl) = \int_s^t \eta_r' \psi(X_r^\epl, Z_r^\epl) \dd r + \sqrt{2\sigma} \int_s^t \eta_r \frac{\partial}{\partial x} \psi(X_r^\epl, Z_r^\epl) \dd W_r + \int_s^t \mathcal L^\epl \psi^\epl(X_r^\epl, Z_r^\epl) \dd W_r.
\end{equation}
Finally, using the fact that $\psi^\epl$ solves the Poisson problem \eqref{eq:poisson_problem} and rearranging the terms we obtain the desired result.
\end{proof}

By an analogous argument, we also obtain the following lemma.
\begin{lemma} \label{lem:decomposition_poisson_2}
Under \cref{as:proof}, there exists a function $\tilde\psi = \tilde\psi(x,z)$, which is bounded along with its derivatives, such that the following decomposition holds true for all $s,t \in [0,T]$
\begin{equation}
\begin{aligned}
\int_s^t \zeta_r v'(Z_r^\epl) p'\left(\frac{X_r^\epl}{\epl}\right) \dd r &= \zeta_t \tilde\psi(X_t^\epl, Z_t^\epl) - \zeta_s \tilde\psi(X_s^\epl, Z_s^\epl) + \int_s^t \zeta_r \dd r \E\left[v'(Z^\epl) p'\left(\frac{X^\epl}{\epl}\right)\right] \\
&\quad - \int_s^t \zeta_r' \tilde\psi(X_r^\epl, Z_r^\epl) \dd r - \sqrt{2\sigma} \int_s^t \zeta_r \frac{\partial}{\partial x} \tilde\psi(X_r^\epl, Z_r^\epl) \dd W_r,
\end{aligned}
\end{equation}
where $\zeta_r$ is defined as
\begin{equation}
\zeta_r := (\beta+r)^{\gamma\E[v'(Z^\epl)v'(X^\epl)]-1}.
\end{equation}
\end{lemma}

In the next section we employ these technical results to prove the main convergence theorem.

\subsection{Proof of the main theorem}

Let us recall the definition of the SGDCT estimator $\widehat A^\epl_T$ in \eqref{eq:estimator_proof} and consider its decomposition
\begin{equation} \label{eq:estimator_decomposition}
\widehat A^\epl_T = \alpha + Q^{\epl, (1)}_T + Q^{\epl, (2)}_T + Q^{\epl, (3)}_T,
\end{equation}
where
\begin{equation}
\begin{aligned}
Q^{\epl, (1)}_T &= (A_0 - \alpha) e^{- \int_0^T \eta_s v'(Z_s^\epl) v'(X_s^\epl) \dd s}, \\
Q^{\epl, (2)}_T &= - \sqrt{2\sigma} \int_0^T \eta_t e^{- \int_t^T \eta_s v'(Z_s^\epl) v'(X_s^\epl) \dd s} v'(Z_t^\epl) \dd W_t, \\
Q^{\epl, (3)}_T &= \frac1\epl \int_0^T \eta_t e^{- \int_t^T \eta_s v'(Z_s^\epl) v'(X_s^\epl) \dd s} v'(Z_t^\epl) p' \left( \frac{X_t^\epl}\epl \right) \dd t. 
\end{aligned}
\end{equation}
In the next lemmas we study the limit as $T \to \infty$ of the three quantities separately. Then, we collect the results together to conclude the proof of \cref{thm:main}.

\begin{lemma} \label{lem:Q1}
Under \cref{as:proof}, it holds 
\begin{equation}
\lim_{T \to \infty} e^{- \int_0^T \eta_s v'(Z_s^\epl) v'(X_s^\epl) \dd s} = 0, \qquad \text{in } L^2.
\end{equation}
\end{lemma}
\begin{proof}
First, letting $\psi$ be given by \cref{lem:decomposition_poisson}, we have
\[
    \left|\eta_T \psi(X_T^\epl, Z_T^\epl) - \eta_0\psi(X_0^\epl, Z_0^\epl)\right| \le \sup |\psi| (\eta_T + \eta_0) \le C < \infty
\]
for some $C>0$ which will change throughout the proof. Furthermore,
\[
    \left|\int_0^T \eta_r' \psi(X_r^\epl, Z_r^\epl)\dd{r}\right| \le \int_0^T |\eta_r'|\dd{r}\cdot\sup|\psi|\le C < \infty,
\]
and
\[
    \E\left[\left(\int_0^T \eta_r\frac{\partial}{\partial x}\psi(X_r^\epl, Z_r^\epl)\dd{W_r}\right)\right] = \E\left[\int_0^T \eta_r^2\left(\frac{\partial}{\partial x}\psi(X_r^\epl, Z_r^\epl)\right)^2\dd{r}\right] \le C\int_0^T \eta_r^2\dd{r} < \infty,
\]
where we have used the It\^{o} isometry and the fact that partial derivatives of $\psi$ are bounded. Using these bounds, we obtain
\begin{equation}
\begin{split}
    &\E\left[\exp\left(-\int_0^T \eta_sv'(Z_s^\epl)v'(X_s^\epl)\dd{s}\right)^2\right]\\
    &\le C\E\left[\exp\left(-2\E[v'(X^\epl)v'(Z^\epl)]\int_0^T\eta_r\dd{r} + 2\sqrt{2\sigma}\int_0^T \eta_r\frac{\partial}{\partial x}\psi(X_r^\epl,Z_r^\epl)\dd{W_r}\right)\right]\\
    &= C\exp\left(-2\E[v'(X^\epl)v'(Z^\epl)]\int_0^T\eta_r\dd{r}\right)\E\left[M_T\exp\left(4\sigma\int_0^T \eta_r^2\left(\frac{\partial}{\partial x} \psi(X_r^\epl, Z_r^\epl)\right)^2\dd{r}\right)\right],
\end{split}
\end{equation}
where
\begin{equation} \label{eq:martingale}
    M_T := \exp\left(2\sqrt{2\sigma}\int_0^T \eta_r \frac{\partial}{\partial x} \psi(X_r^\epl, Z_r^\epl) \dd{W_r} - 4\sigma \int_0^T \eta_r^2 \left(\frac{\partial}{\partial x} \psi(X_r^\epl, Z_r^\epl)\right)^2\dd{r}\right)
\end{equation}
is an exponential martingale. Using the fact that $M_T$ is a martingale and that $\int_0^T \eta_r^2\left(\frac{\partial}{\partial x} \psi(X_r^\epl, Z_r^\epl)\right)^2\dd{r} < \infty$ since the partial derivatives of $\psi$ are bounded, we then have
\begin{equation}
\begin{split}
    \E\left[\exp\left(-\int_0^T \eta_sv'(Z_s^\epl)v'(X_s^\epl)\dd{s}\right)^2\right] &\le C\exp\left(-2\E[v'(X^\epl)v'(Z^\epl)]\int_0^T\eta_r\dd{r}\right) \E[M_T]\\ 
    &= C\exp\left(-2\E[v'(X^\epl)v'(Z^\epl)]\int_0^T\eta_r\dd{r}\right),
\end{split}
\end{equation}
which converges to $0$ since $E[v'(X^\epl)v'(Z^\epl)] > 0$ by \eqref{eq:as_epsilon} and $\int_0^T \eta_r\dd{r} \to \infty$.
\end{proof}

\begin{lemma} \label{lem:Q2}
Under \cref{as:proof}, it holds 
\begin{equation}
\lim_{T \to \infty} \int_0^T \eta_t e^{- \int_t^T \eta_s v'(Z_s^\epl) v'(X_s^\epl) \dd s} v'(Z_t^\epl) \dd W_t = 0, \qquad \text{in } L^2.
\end{equation}
\end{lemma}
\begin{proof}
We have by the It\^{o} isometry that
\begin{equation}
    \begin{split}
        \E\left[\left(\int_0^T \eta_t e^{-\int_t^T \eta_s v'(Z_s^\epl)v'(X_s^\epl)\dd{s}}v'(Z_t^\epl)\dd{W_t}\right)^2\right] &= \E\left[\int_0^T \eta_t^2 e^{-2\int_t^T \eta_s v'(Z_s^\epl)v'(X_s^\epl)\dd{s}}v'(Z_t^\epl)^2\dd{t}\right].
    \end{split}
\end{equation}
Since $v'(Z_t^\epl)$ is bounded, we have by Tonelli's theorem and computations in the proof of \cref{lem:Q1} that the right hand side above is bounded from above by
\begin{equation}
\begin{split}
    &C\int_0^T \eta_t^2 \E\left[\exp\left(-2\int_t^T \eta_r\dd{r}\cdot\E[v'(Z^\epl)v'(X^\epl)] + 2\sqrt{2\sigma}\int_t^T \eta_r\frac{\partial}{\partial x}\psi(X_r^\epl,Z_r^\epl)\dd{W_r}\right)\right]\dd{t}\\
    &= C\int_0^T \eta_t^2 \exp\left(-2\int_t^T \eta_r\dd{r}\cdot\E[v'(Z^\epl)v'(X^\epl)]\right) \E\left[\frac{M_T}{M_t}\exp\left(-4\sigma\int_t^T \eta_r^2\left(\frac{\partial}{\partial x}\psi(X_r^\epl, Z_r^\epl)\right)^2\dd{r}\right)\right]\dd{t}\\
    &\le C\int_0^T \eta_t^2 \exp\left(-2\int_t^T \eta_r\dd{r}\cdot\E[v'(Z^\epl)v'(X^\epl)]\right) \E\left[\frac{M_T}{M_t}\right]\dd{t}\\
    &= C\int_0^T \eta_t^2 \exp\left(-2\int_t^T \eta_r\dd{r}\cdot\E[v'(Z^\epl)v'(X^\epl)]\right)\dd{t},
\end{split}
\end{equation}
where the martingale $M_t$ is defined in \eqref{eq:martingale} and the last line follows since
\[
    \E\left[\frac{M_T}{M_t}\right] = \E\left[\E\left[\frac{M_T}{M_t}\mid M_t\right]\right] = \E\left[\frac{1}{M_t}\E[M_T\mid M_t]\right] = \E\left[\frac{1}{M_t}M_t\right] = 1.
\]
Now substituting the expressions for $\eta_r$ and $\eta_t$ gives
\begin{equation}
    \begin{split}
        &\E\left[\left(\int_0^T \eta_t e^{-\int_t^T \eta_s v'(Z_s^\epl)v'(X_s^\epl)\dd{s}}v'(Z_t^\epl)\dd{W_t}\right)^2\right]\\
            &\le C\int_0^T \eta_t^2 \exp\left(-2\int_t^T \eta_r\dd{r}\cdot\E[v'(Z^\epl)v'(X^\epl)]\right)\dd{t}\\
            &= C\int_0^T \left(\frac{\gamma}{\beta+t}\right)^2\exp\left(-2\E[v'(Z^\epl)v'(X^\epl)]\int_t^T \frac{\gamma}{\beta+r}\dd{r}\cdot\right)\dd{t}\\
            &= C\int_0^T \left(\frac{\gamma}{\beta+t}\right)^2\left(\frac{\beta+T}{\beta+t}\right)^{-2\gamma\E[v'(Z^\epl)v'(X^\epl)]}\dd{t},\\
    \end{split}
\end{equation}
which implies
\begin{equation}
    \begin{split}
        &\E\left[\left(\int_0^T \eta_t e^{-\int_t^T \eta_s v'(Z_s^\epl)v'(X_s^\epl)\dd{s}}v'(Z_t^\epl)\dd{W_t}\right)^2\right]\\
            &\le C\gamma^2(\beta+T)^{-2\gamma\E[v'(Z^\epl)v'(X^\epl)]}\int_0^T \frac{1}{(\beta+t)^{2-2\gamma\E[v'(Z^\epl)v'(X^\epl)]}}\dd{t}\\
            &= \mathcal O(T^{-2\gamma\E[v'(Z^\epl)v'(X^\epl)] + 2\gamma\E[v'(Z^\epl)v'(X^\epl)] - 2 + 1})\\
            &= \mathcal O(T^{-1}),
    \end{split}
\end{equation}
where in computing the integral, we used the fact that $\gamma\E[v'(Z^\epl)v'(X^\epl)] > 1$ by \eqref{eq:as_epsilon}. This shows the desired limit as $T\to\infty$.
\end{proof}

\begin{lemma} \label{lem:Q3}
Under \cref{as:proof}, it holds 
\begin{equation}
\lim_{T \to \infty} \int_0^T \eta_t e^{- \int_t^T \eta_s v'(Z_s^\epl) v'(X_s^\epl) \dd s} v'(Z_t^\epl) p' \left( \frac{X_t^\epl}\epl \right) \dd t = \frac{\E \left[ v'(Z^\epl) p' \left( \frac{X^\epl}\epl \right) \right]}{\E \left[ v'(Z^\epl) v'(X^\epl) \right]}, \qquad \text{in } L^2.
\end{equation}
\end{lemma}
\begin{proof}
First, we have by the triangle inequality that
\begin{equation}
\begin{split}
    &\left|\int_0^T \eta_t\operatorname{exp}\left(-\int_t^T \eta_sv'(Z_s^\epl)v'(X_s^\epl)\dd{s}\right)v'(Z_t^\epl)p'\left(\frac{X_t^\epl}{\epl}\right)\d{t} - \frac{\E\left[v'(Z^\epl)p'(\frac{X^\epl}{\epl})\right]}{\E\left[v'(Z^\epl)v'(X^\epl)\right]}\right|\\
    &\le \underbrace{\left|\int_0^T \eta_t v'(Z_t^\epl)p'\left(\frac{X_t^\epl}{\epl}\right)\left[\exp\left(-\int_t^T \eta_sv'(Z_s^\epl)v'(X_s^\epl)\dd{s}\right) - \exp\left(-\int_t^T \eta_s\dd{s}\cdot\E[v'(Z^\epl)v'(X^\epl)]\right)\right]\dd{t}\right|}_{I}\\
    &+ \underbrace{\left| \int_0^T \eta_t \left(v'(Z_t^\epl)p'\left(\frac{X_t^\epl}{\epl}\right) - \E\left[v'(Z^\epl)p'\left(\frac{X^\epl}{\epl}\right)\right]\right)\exp\left(-\int_t^T \eta_s\dd{s}\cdot\E[v'(Z^\epl)v'(X^\epl)]\right)\dd{t}\right|}_{II}\\
    &+ \underbrace{\left|\int_0^T \eta_t \E\left[v'(Z^\epl)p'\left(\frac{X^\epl}{\epl}\right)\right]\exp\left(-\int_t^T \eta_s\dd{s}\cdot\E[v'(Z^\epl)v'(X^\epl)]\right)\dd{t} - \frac{\E[v'(Z^\epl)p'(\frac{X^\epl}{\epl})]}{\E[v'(Z^\epl)V'(X^\epl)]}\right|}_{III},
\end{split}
\end{equation}
where $(X^\epl, Z^\epl)$ is distributed according to the invariant measure $\mu^\epl$ for the joint multiscale process $(X_t^\epl, Z_t^\epl)$.
We now show that $I\to 0$ in $L^2$. We have 
\begin{equation}
    \begin{split}
        &\mathfrak{I}(T)
        := \left|\int_0^T \eta_t v'(Z_t^\epl)p'\left(\frac{X_t^\epl}{\epl}\right)\left[e^{-\int_t^T \eta_sv'(Z_s^\epl)v'(X_s^\epl)\dd{s}} - e^{-\int_t^T \eta_s\dd{s}\cdot\E[v'(Z^\epl)v'(X^\epl)]}\right]\dd{t}\right|^2\\
        &= \left|\int_0^T \eta_t v'(Z_t^\epl)p'\left(\frac{X_t^\epl}{\epl}\right)e^{-\int_t^T \eta_s\dd{s}\cdot\E[v'(Z^\epl)v'(X^\epl)]}\left(e^{-\int_t^T \eta_sv'(Z_s^\epl)v'(X_s^\epl)\dd{s} + \int_t^T \eta_s\dd{s}\cdot\E[v'(Z^\epl)v'(X^\epl)]} - 1\right)\dd{t}\right|^2\\
        &\le CT \int_0^T \eta_t^2 e^{-2\int_t^T \eta_s\dd{s}\cdot\E[v'(Z^\epl)v'(X^\epl)]}\Bigg|e^{-\int_t^T \eta_sv'(Z_s^\epl)v'(X_s^\epl)\dd{s} + \int_t^T \eta_s\dd{s}\cdot\E[v'(Z^\epl)v'(X^\epl)]} - 1\Bigg|^2\dd{t},
    \end{split}
\end{equation}
where the last line follows from H\"{o}lder's inequality and the fact that $p'$ and $V'$ are bounded on the torus. Now applying \cref{lem:decomposition_poisson}, we have
\begin{equation}
    \begin{split}
        \mathfrak{I}(T)
            &\le CT \int_0^T \eta_t^2 e^{-2\int_t^T \eta_s\dd{s}\cdot\E[v'(Z^\epl)v'(X^\epl)]}\\
            &\hspace{20ex}\left|e^{-\eta_T\psi(X_T^\epl, Z_T^\epl) + \eta_t\psi(X_t^\epl,Z_t^\epl) + \int_t^T \eta_s'\psi(X_s^\epl, Z_s^\epl)\dd{s} + \sqrt{2\sigma}\int_t^T \eta_s\frac{\partial}{\partial x}\psi(X_s^\epl, Z_s^\epl)\dd{W_s}} - 1\right|^2\dd{t}\\
            &= CT \int_0^T \left(\frac{\gamma}{\beta+t}\right)^2\left(\frac{\beta+T}{\beta+t}\right)^{-2\gamma\E[v'(Z^\epl)v'(X^\epl)]}\left|e^{Q(t,T)} - 1\right|^2\dd{t},
    \end{split}
\end{equation}
where
\[
    Q(t, T) = -\eta_T\psi(X_T^\epl, Z_T^\epl) + \eta_t\psi(X_t^\epl,Z_t^\epl) + \int_t^T \eta_s'\psi(X_s^\epl, Z_s^\epl)\dd{s} + \sqrt{2\sigma}\int_t^T \eta_s\frac{\partial}{\partial x}\psi(X_s^\epl, Z_s^\epl)\dd{W_s}.
\]
We focus on bounding $|e^{Q(t,T)}-1|^2 = e^{2Q(t,T)} - 2e^{Q(t,T)}+1$. First, recall that
\[
    M_T := \exp\left(2\sqrt{2\sigma}\int_0^T \eta_r \frac{\partial}{\partial x} \psi(X_r^\epl, Z_r^\epl) \dd{W_r} - 4\sigma \int_0^T \eta_r^2 \left(\frac{\partial}{\partial x} \psi(X_r^\epl, Z_r^\epl)\right)^2\dd{r}\right)
\]
and
\[
    \tilde M_T := \exp\left(\sqrt{2\sigma}\int_0^T \eta_r \frac{\partial}{\partial x} \psi(X_r^\epl, Z_r^\epl) \dd{W_r} - \sigma \int_0^T \eta_r^2 \left(\frac{\partial}{\partial x} \psi(X_r^\epl, Z_r^\epl)\right)^2\dd{r}\right)
\]
are exponential martingales. Then repeatedly using the fact that $t\le T$ and that $\psi$ and its derivatives are bounded, we have
\begin{equation}
    \begin{split}
        e^{2Q(t,T)}
            &= e^{-\frac{2\gamma}{\beta+T}\psi(X_T^\epl, Z_T^\epl) + \frac{2\gamma}{\beta+t}\psi(X_t^\epl, Z_t^\epl) - 2\int_t^T \frac{\gamma}{(\beta+s)^2}\psi(X_s^\epl, Z_s^\epl)\dd{s} + 2\sqrt{2\sigma}\int_t^T \eta_s\frac{\partial}{\partial x}\psi(X_s^\epl, Z_s^\epl)\dd{W_s}}\\
            &\le \exp\left(\frac{C}{\beta+T} + \frac{C}{\beta+t} + C\int_t^T \frac{1}{(\beta+s)^2}\dd{s} + 2\sqrt{2\sigma}\int_t^T \eta_s\frac{\partial}{\partial x}\psi(X_s^\epl, Z_s^\epl)\dd{W_s}\right)\\
            &\le e^{\frac{C}{\beta+t}} \frac{M_T}{M_t} \exp\left(-4\sigma\int_t^T \eta_s^2\left(\frac{\partial}{\partial x}\psi(X_s^\epl, Z_s^\epl)\right)^2\dd{s}\right)\\
            &\le e^{\frac{C}{\beta+t}}\frac{M_T}{M_t}.
    \end{split}
\end{equation}
Similarly,
\begin{equation}
    \begin{split}
        e^{Q(t,T)}
            &\ge e^{-\frac{C}{\beta+T} - \frac{C}{\beta+t} - C\int_t^T \frac{1}{(\beta+s)^2}\dd{s} + \sqrt{2\sigma}\int_t^T \eta_s\frac{\partial}{\partial x}\psi(X_s^\epl,Z_s^\epl)\dd{W_s}}\\
            &\ge e^{-\frac{C}{\beta+t}}\frac{\tilde M_T}{\tilde M_t}\exp\left(-\sigma\int_t^T \eta_s^2\left(\frac{\partial}{\partial x}\psi(X_s^\epl, Z_s^\epl)\right)^2\dd{s}\right)\\
            &\ge e^{-\frac{C}{\beta+t}}\frac{\tilde M_T}{\tilde M_t}\exp\left(-C\int_t^T \left(\frac{\gamma}{\beta+s}\right)^2\dd{s}\right)\\
            &\ge e^{-\frac{C}{\beta+t}}\frac{\tilde M_T}{\tilde M_t}.
    \end{split}
\end{equation}
Thus,
\[
    \left|e^{Q(t,T)} - 1\right|^2 \le e^{\frac{C}{\beta+t}}\frac{M_T}{M_t} - 2e^{-\frac{C}{\beta+t}}\frac{\tilde M_T}{\tilde M_t} + 1.
\]
It follows from Tonelli's theorem and the fact that $\E[M_T/M_t] = \E[\tilde M_T/\tilde M_t] = 1$ that
\begin{equation}
    \begin{split}
        \E\mathfrak{I}(T) 
            &\le CT\E\int_0^T \left(\frac{\gamma}{\beta+t}\right)^2\left(\frac{\beta+T}{\beta+t}\right)^{-2\gamma\E[v'(Z^\epl)v'(X^\epl)]}\left|e^{Q(t,T)} - 1\right|^2\dd{t}\\
            &= CT\int_0^T \left(\frac{\gamma}{\beta+t}\right)^2\left(\frac{\beta+T}{\beta+t}\right)^{-2\gamma\E[v'(Z^\epl)v'(X^\epl)]}\E\left[\left|e^{Q(t,T)} - 1\right|^2\right]\dd{t}\\
            &\le CT\int_0^T \left(\frac{\gamma}{\beta+t}\right)^2\left(\frac{\beta+T}{\beta+t}\right)^{-2\gamma\E[v'(Z^\epl)v'(X^\epl)]}\left(e^{\frac{C}{\beta+t}} - 2e^{-\frac{C}{\beta+t}}+1\right)\dd{t}.
    \end{split}
\end{equation}
Now since $\gamma\E[v'(Z^\epl)v'(X^\epl)] > 1$, there is $\lambda>0$ such that $2\gamma\E[v'(Z^\epl)v'(X^\epl)] - 2 - \lambda > 0$, so because $T\ge t$, we have
\[
    \left(\frac{\gamma}{\beta+t}\right)^2\left(\frac{\beta+T}{\beta+t}\right)^{-2\gamma\E[v'(Z^\epl)v'(X^\epl)]} = \gamma^2 \frac{(\beta+t)^{2\gamma\E[v'(Z^\epl)v'(X^\epl)] - 2-\lambda+\lambda}}{(\beta+T)^{2\gamma\E[v'(Z^\epl)v'(X^\epl)]-2-\lambda+2+\lambda}} \le \frac{\gamma^2(\beta+t)^\lambda}{(\beta+T)^{2+\lambda}}.
\]
Thus, it follows that
\[
    \E\mathfrak{I}(T) \le \frac{CT}{(\beta+T)^{2+\lambda}}\int_0^T (\beta+t)^\lambda\left(e^{\frac{C}{\beta+t}} - 2e^{-\frac{C}{\beta+t}}+1\right)\dd{t}.
\]
Now we have that 
\begin{equation}
\begin{split}
    e^{\frac{C}{\beta+t}} - 2e^{-\frac{C}{\beta+t}}+1 &= 1 + \frac{C}{\beta+t} + o((\beta+t)^{-1}) - 2\left(1 - \frac{C}{\beta+t} + o((\beta+t)^{-1})\right) + 1\\
    &= \frac{3C}{\beta+t} + o((\beta+t)^{-1})
\end{split}
\end{equation}
so that for $t_0, T$ sufficiently large, we have
\[
    \E\mathfrak{I}(T) \le \frac{CT}{(\beta+T)^{2+\lambda}} + \frac{CT}{(\beta+T)^{2+\lambda}}\int_{t_0}^T (\beta+t)^{-1+\lambda}\dd{t} = \frac{CT}{(\beta+T)^{2+\lambda}} + \frac{CT(\beta+T)^\lambda}{(\beta+T)^{2+\lambda}}.
\]
In particular, $I\to 0$ in $L^2$ at a rate of $\mathcal O(T^{-1/2})$ as $T\to\infty$. Now we look at $II$. As preparation toward this end, by \cref{lem:decomposition_poisson_2}, we have
\begin{equation}
\begin{aligned}
\int_0^T \zeta_r v'(Z_r^\epl) p'\left(\frac{X_r^\epl}{\epl}\right) \dd r &= \zeta_T \tilde\psi(X_T^\epl, Z_T^\epl) - \zeta_0 \tilde\psi(X_0^\epl, Z_0^\epl) + \int_0^T \zeta_r \dd r \E\left[v'(Z^\epl) p'\left(\frac{X^\epl}{\epl}\right)\right] \\
&\quad - \int_0^T \zeta_r' \tilde\psi(X_r^\epl, Z_r^\epl) \dd r - \sqrt{2\sigma} \int_0^T \zeta_r \frac{\partial}{\partial x} \tilde\psi(X_r^\epl, Z_r^\epl) \dd W_r,
\end{aligned}
\end{equation}
where $\zeta_r := (\beta+r)^{\gamma\E[v'(Z^\epl)v'(X^\epl)]-1}$. Rearranging and dividing by
\begin{equation}
    \begin{split}
        \int_0^T \zeta_r\dd{r} &= \int_0^T (\beta+r)^{\gamma\E[v'(Z^\epl)v'(X^\epl)]-1}\dd{r}\\
        &= \frac{1}{\gamma\E[v'(Z^\epl)v'(X^\epl)]}\left[(\beta+T)^{\gamma\E[v'(Z^\epl)v'(X^\epl)]} - \beta^{\gamma\E[v'(Z^\epl)v'(X^\epl)]}\right] = \mathcal O(T^{\gamma\E[v'(Z^\epl)v'(X^\epl)]})
    \end{split}
\end{equation}
gives
\begin{equation}
    \begin{split}
        &\left(\int_0^T \zeta_r\dd{r}\right)^{-1}\int_0^T \zeta_r\left(v'(Z_r^\epl)p'\left(\frac{X_r^\epl}{\epl}\right) - \E\left[v'(Z^\epl)p'\left(\frac{X^\epl}{\epl}\right)\right]\right)\dd{r}\\
        &\hspace{10ex}= \underbrace{\frac{\zeta_T\tilde\psi(X_T^\epl, Z_T^\epl) - \zeta_0\tilde\psi(X_0^\epl, Z_0^\epl)}{\int_0^T \zeta_r\dd{r}}}_{(i)} - \underbrace{\frac{\int_0^T \zeta_r'\tilde\psi(X_r^\epl,Z_r^\epl)\dd{r}}{\int_0^T \zeta_r\dd{r}}}_{(ii)} - \underbrace{\frac{\sqrt{2\sigma}\int_0^T \zeta_r\frac{\partial}{\partial x}\tilde\psi(X_r^\epl, Z_r^\epl)\dd{W_r}}{\int_0^T \zeta_r\dd{r}}}_{(iii)}.
    \end{split}
\end{equation}
For $(i)$, because $\tilde\psi$ is bounded, we have
\[
    \left|\frac{\zeta_T\tilde\psi(X_T^\epl, Z_T^\epl) - \zeta_0\tilde\psi(X_0^\epl, Z_0^\epl)}{\int_0^T \zeta_r\dd{r}}\right| \le \frac{\mathcal O(T^{\gamma\E[v'(Z^\epl)v'(X^\epl)]-1})}{\mathcal O(T^{\gamma\E[v'(Z^\epl)v'(X^\epl)]})} + \frac{C}{\mathcal O(T^{\gamma\E[v'(Z^\epl)v'(X^\epl)]})} = \mathcal O(T^{-1}),
\]
so $(i) \to 0$ in $L^2$ at a rate of $\mathcal O(T^{-1})$ as $T\to\infty$. Similarly, we have
\[
    \left|\frac{\int_0^T \zeta_r'\tilde\psi(X_r^\epl, Z_r^\epl)}{\int_0^T \zeta_r\dd{r}}\right| \le \frac{\mathcal O(T^{\gamma\E[v'(Z^\epl)v'(X^\epl)]-1})}{\mathcal O(T^{\gamma\E[v'(Z^\epl)v'(X^\epl)]})} = \mathcal{O}(T^{-1}),
\]
so $(ii)\to 0$ in $L^2$ at a rate of $\mathcal O(T^{-1})$ as well. For $(iii)$, we have by It\^{o} isometry and the fact that the derivative of $\psi$ is bounded that
\begin{equation}
    \begin{split}
        &2\sigma\left(\int_0^T \zeta_r\dd{r}\right)^{-2}\E\left[\left(\int_0^T \zeta_r\frac{\partial}{\partial x}\tilde\psi(X_r^\epl,Z_r^\epl)\dd{W_r}\right)^2\right]\\
        &\hspace{10ex}= 2\sigma\left(\int_0^T \zeta_r\dd{r}\right)^{-2}\E\left[\int_0^T \zeta_r^2\left(\frac{\partial}{\partial x}\tilde\psi(X_r^\epl,Z_r^\epl)\right)^2\dd{r}\right]\\
        &\hspace{10ex}\le C\left(\int_0^T \zeta_r\dd{r}\right)^{-2}\int_0^T \zeta_r^2\dd{r} = \mathcal O(T^{-1}).
    \end{split}
\end{equation}
Thus, we see that $(iii)\to 0$ in $L^2$ at a rate of $\mathcal O(T^{-1/2})$ as $T\to\infty$. It follows that 
\[
    \left(\int_0^T \zeta_r\dd{r}\right)^{-1}\int_0^T \zeta_r\left(v'(Z_r^\epl)p'\left(\frac{X_r^\epl}{\epl}\right) - \E\left[v'(Z^\epl)p'\left(\frac{X^\epl}{\epl}\right)\right]\right)\dd{r} \to 0\quad\text{in $L^2$}
\]
at a rate of $\mathcal O(T^{-1/2})$. Now returning to the convergence of $II$, we have
\begin{equation}
    \begin{split}
        &\E\left[\left| \int_0^T \eta_t \left(v'(Z_t^\epl)p'\left(\frac{X_t^\epl}{\epl}\right) - \E\left[v'(Z^\epl)p'\left(\frac{X^\epl}{\epl}\right)\right]\right)\exp\left(-\int_t^T \eta_s\dd{s}\cdot\E[v'(Z^\epl)v'(X^\epl)]\right)\dd{t}\right|^2\right]\\
        &= \E\left[\left| \int_0^T \left(\frac{\gamma}{\beta+t}\right)\left(\frac{\beta+T}{\beta+t}\right)^{-\gamma\E[v'(Z^\epl)v'(X^\epl)]} \left(v'(Z_t^\epl)p'\left(\frac{X_t^\epl}{\epl}\right) - \E\left[v'(Z^\epl)p'\left(\frac{X^\epl}{\epl}\right)\right]\right)\dd{t}\right|^2\right]\\
        &= \frac{\gamma^2}{(\beta+T)^{2\gamma\E[v'(Z^\epl)v'(X^\epl)]}}\E\left[\left|\int_0^T\zeta_t\left(v'(Z_t^\epl)p'\left(\frac{X_t^\epl}{\epl}\right) - \E\left[v'(Z^\epl)p'\left(\frac{X^\epl}{\epl}\right)\right]\right)\dd{t}\right|^2\right]\\
        &= \frac{\gamma^2\left(\int_0^T \zeta_t\dd{t}\right)^2}{(\beta+T)^{2\gamma\E[v'(Z^\epl)v'(X^\epl)]}}\E\left[\left|\left(\int_0^T \zeta_t\dd{t}\right)^{-1}\int_0^T\zeta_t\left(v'(Z_t^\epl)p'\left(\frac{X_t^\epl}{\epl}\right) - \E\left[v'(Z^\epl)p'\left(\frac{X^\epl}{\epl}\right)\right]\right)\dd{t}\right|^2\right]\\
        &\le \frac{\gamma^2\left(\int_0^T \zeta_t\dd{t}\right)^2\cdot\mathcal O(T^{-1})}{(\beta+T)^{2\gamma\E[v'(Z^\epl)v'(X^\epl)]}} = \frac{\mathcal O(T^{2\gamma\E[v'(Z^\epl)v'(X^\epl)]})\cdot \mathcal O(T^{-1})}{\mathcal O(T^{2\gamma\E[v'(Z^\epl)v'(X^\epl)]})} = \mathcal{O}(T^{-1})
    \end{split}
\end{equation}
so that $II\to 0$ in $L^2$ at a rate of $\mathcal O(T^{-1/2})$. We now turn to showing that $III$ converges to $0$ in $L^2$. We have
\begin{equation}
    \begin{split}
        &\E\left[\left|\int_0^T \eta_t \E\left[v'(Z^\epl)p'\left(\frac{X^\epl}{\epl}\right)\right]\exp\left(-\int_t^T \eta_s\dd{s}\cdot\E[v'(Z^\epl)v'(X^\epl)]\right)\dd{t} - \frac{\E[v'(Z^\epl)p'(\frac{X^\epl}{\epl})]}{\E[v'(Z^\epl)V'(X^\epl)]}\right|^2\right]\\
        &= \left|\int_0^T \left(\frac{\gamma}{\beta+t}\right)\left(\frac{\beta+T}{\beta+t}\right)^{-\gamma\E[v'(Z^\epl)v'(X^\epl)]} \E\left[v'(Z^\epl)p'\left(\frac{X^\epl}{\epl}\right)\right]\dd{t} - \frac{\E[v'(Z^\epl)p'(\frac{X^\epl}{\epl})]}{\E[v'(Z^\epl)V'(X^\epl)]}\right|^2\\
        &= \left|\frac{\gamma\E\left[v'(Z^\epl)p'\left(\frac{X^\epl}{\epl}\right)\right]}{(\beta+T)^{\gamma\E[v'(Z^\epl)v'(X^\epl)]}}\int_0^T (\beta+t)^{\gamma\E[v'(Z^\epl)v'(X^\epl)]-1} \dd{t} - \frac{\E[v'(Z^\epl)p'(\frac{X^\epl}{\epl})]}{\E[v'(Z^\epl)V'(X^\epl)]}\right|^2\\
        &= \left|\frac{\E\left[v'(Z^\epl)p'\left(\frac{X^\epl}{\epl}\right)\right]}{\E[v'(Z^\epl)v'(X^\epl)]}\left(\frac{(\beta+T)^{\gamma\E[v'(Z^\epl)v'(X^\epl)]} - \beta^{\gamma\E[v'(Z^\epl)v'(X^\epl)]}}{(\beta+T)^{\gamma\E[v'(Z^\epl)v'(X^\epl)]}}\right) - \frac{\E[v'(Z^\epl)p'(\frac{X^\epl}{\epl})]}{\E[v'(Z^\epl)V'(X^\epl)]}\right|^2,
    \end{split}
\end{equation}
which we see converges to $0$ at the rate $\mathcal O(T^{-2\gamma\E[v'(Z^\epl)v'(X^\epl)]})$. Thus, since $\gamma\E[v'(Z^\epl)v'(X^\epl)] > 1$, we have that $III\to 0$ in $L^2$ at a rate of $\mathcal O(T^{-1})$. Combining the convergence for $I$, $II$, and $III$ gives
\[
    \int_0^T \eta_t\operatorname{exp}\left(-\int_t^T \eta_sv'(Z_s^\epl)v'(X_s^\epl)\dd{s}\right)v'(Z_t^\epl)p'\left(\frac{X_t^\epl}{\epl}\right)\d{t} - \frac{\E\left[v'(Z^\epl)p'(\frac{X^\epl}{\epl})\right]}{\E\left[v'(Z^\epl)v'(X^\epl)\right]} \to 0 \quad\text{in $L^2$}
\]
at a rate of $\mathcal O(T^{-1/2})$, as desired.
\end{proof}

\begin{proof}[Proof of \cref{thm:main}]
By the decomposition \eqref{eq:estimator_decomposition} and due to \cref{lem:Q1,lem:Q2,lem:Q3} we have
\begin{equation}
\lim_{T \to \infty} \widehat A^\epl_T = \alpha + \frac{\E \left[ v'(Z^\epl) p' \left( \frac{X^\epl}\epl \right) \right]}{\epl \E \left[ v'(Z^\epl) v'(X^\epl) \right]}, \qquad \text{in } L^2.
\end{equation}
Finally, applying \cref{lem:magic_formula} we obtain
\begin{equation}
\lim_{\epl \to 0} \lim_{T \to \infty} \widehat A^\epl_T = \alpha + A - \alpha = A, \qquad \text{in } L^2,
\end{equation}
which is the desired result.
\end{proof}

\begin{remark} \label{rem:rate}
From the proof of \cref{thm:main}, and in particular from the proofs of \cref{lem:Q1,lem:Q2,lem:Q3}, we notice that the rate of convergence of the estimator $\widehat A^\epl_T$ with respect to the final time $T$ is $\mathcal O(T^{-1/2})$. On the other hand, the rate of convergence with respect to the multiscale parameter $\epl$ is dependent on the limit in \cref{lem:magic_formula}, whose convergence rate is unknown. Its study would first require the analysis of the rate of weak convergence of the invariant measure $\mu^\epl$ of the multiscale joint process $(X_t^\epl, Z_t^\epl)$ towards the invariant measure $\mu$ of its homogenized counterpart $(X_t, Z_t)$. We will return to this problem in future work.
\end{remark}

\section{Numerical experiments} \label{sec:experiments}
We now present numerical experiments to demonstrate our results. The code to reproduce these experiments is available on Github \cite{PaperCode}. Throughout our test cases, we sample multiscale data generated from equation \eqref{eq:SDE_multiscale} with the Euler-Maruyama scheme
\begin{equation} \label{eq:SDE_Euler_Maruyama_scheme} 
X_{n+1}^\epl - X_n^\epl = -  (\mathcal V^\epl)'(X_n^\epl; \alpha) \Delta t + \sqrt{2\sigma} \Delta W_n,\quad X_0^\epl 
= 0,\quad \Delta W_n \sim \mathcal N(0, \Delta t),
\end{equation}
where the time step $\Delta t$ is chosen sufficiently small with respect to $\epl$, i.e., $\Delta t = \epl^3$. Then, unless otherwise specified, we solve the SDE \eqref{eq:SDE_estimator_filter} for the modified SGDCT estimator using the Euler-Maruyama scheme
\begin{equation} \label{eq:SGDCT_Euler_Maruyama_scheme}
\widehat A^\epl_{n+1} - \widehat A^\epl_n = - \eta_n (U(Z_n^\epl) \otimes U(X_n^\epl)) \widehat A^\epl_n \Delta t - \eta_n U(Z_n^\epl) (X_{n+1}^\epl - X_n^\epl), \quad \widehat A_0^\epl = 0,
\end{equation}
for an appropriate vector of basis functions $U$. Furthermore, we take as learning rate 
\begin{equation}
\eta_n = \frac{10}{10 + n\Delta t},
\end{equation}
i.e., we set $\gamma = \beta = 10$, except where noted. We remark that choosing a good learning rate is highly dependent on the given problem \cite{smith2018disciplinedapproachneuralnetwork}. In our setting, the variance of our estimators depends on the constants appearing in the learning rate, and the optimal value can depend on the unknown parameter $A$ and so cannot be known in practice \cite{SiS20}. Due to these difficulties, we intuitively choose the learning rate to be smaller when the learned weights are not converging as expected.

Now, given a sample path $\{X_n^\epl\}_{n=0}^N$ with $N = \lfloor T/\Delta t\rfloor$, the exponential filtered path $\{Z_n^\epl\}_{n=0}^N$ is computed as
\begin{equation}\label{eq:discrete_exponential_filter}
    Z_0^\epl = 0,\quad\text{and}\quad Z_n^\epl = e^{-\Delta t/\delta} Z_{n-1}^\epl + \frac1\delta e^{-\Delta t/\delta}X_n^\epl\Delta t\quad\text{for}\quad n>0,
\end{equation}
and the moving average path $\{Z_n^\epl\}_{n=0}^N$ is computed as
\begin{equation}\label{eq:discrete_moving_average}
    \begin{cases}
        Z_0^\epl = 0,\\
        Z_n^\epl = (n\Delta t)^{-1}\sum_{i=0}^n X_n^\epl \Delta t &\text{if } 1 \le n < S,\\
        Z_n^\epl = \delta^{-1}\sum_{i=n-S}^{n-1} X_n^\epl\Delta t &\text{if } n \ge S,
    \end{cases}
\end{equation}
where $S = \lfloor \delta/\Delta t\rfloor$.

We now present three experiments. In the first, we will verify our theoretical results by showing that the SGDCT estimator is asymptotically unbiased when we use filtered data and biased otherwise. We also show the dependence on the filter width $\delta$. In our second experiment, we demonstrate the efficacy of our estimator for a two-dimensional drift coefficient, employing both the exponential filter and the moving average. Finally, in our third experiment, we learn the drift function using our semiparametric approach and observe that the learned function provides a satisfactory approximation of the true drift. Because our theoretical results are for the limit as $T\to\infty$, we choose the final time of computation $T$ in each experiment so that we observe convergence when computationally feasible.

\subsection{Verification of the theoretical results} \label{sec:verification}

In this section, we numerically verify our theoretical results. We let
\[
    V(x;\alpha) = \alpha \frac{x^2}{2} \qquad \text{and} \qquad p(x, y) = \sin(y),
\]
so that
\[
    \nabla\mathcal V^\epl(x;\alpha) = \alpha x + \frac1\epl \cos\left(\frac{x}{\epl}\right).
\]
We begin by showing that without using filtered data in our SGDCT estimator, that is, by using the estimator \eqref{eq:SDE_estimator}, we obtain $\alpha$ as a biased estimate of $A$. Thus, instead of solving \eqref{eq:SDE_estimator_filter} to obtain our estimates of $A$, we solve \eqref{eq:SDE_estimator} using the Euler-Maruyama scheme
\begin{equation}
    \widetilde A^\epl_{n+1} - \widetilde A^\epl_n = - \eta_n (X_n^\epl)^2 \widetilde A^\epl_n \Delta t - \eta_n X_n^\epl (X_{n+1}^\epl - X_n^\epl), \quad \widetilde A^\epl_0 = 0.
\end{equation}

\begin{figure}
\begin{tabular}{cc}
\centering
    \includegraphics[scale=0.48]{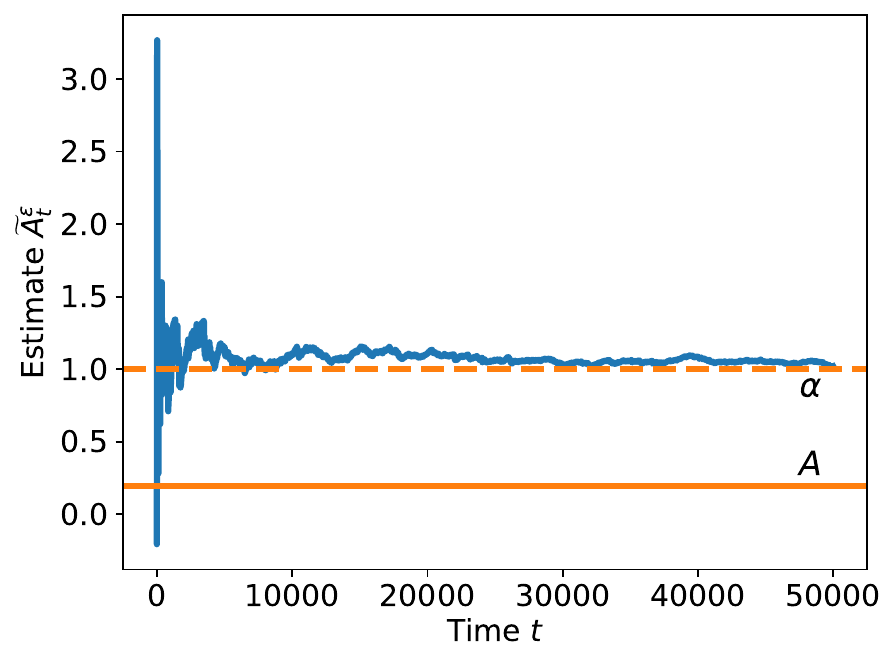} & \includegraphics[scale=0.48]{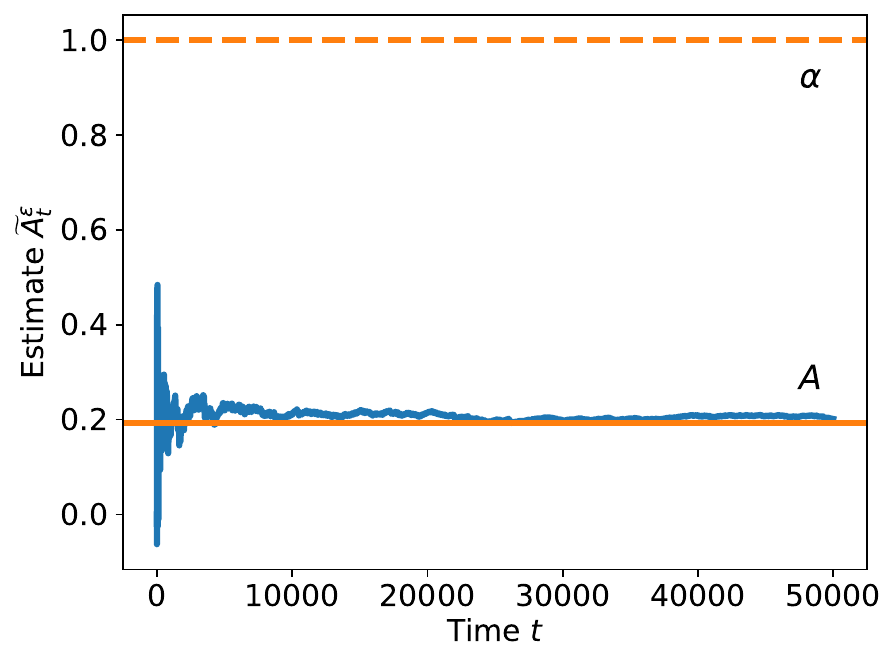}\\
    (a) SGDCT with no filter. & (b) SGDCT with exponential filter.
\end{tabular}
\caption{Comparison between SGDCT estimator (a) without filtered data and (b) with exponential filtered data, for the example in \cref{sec:verification}. In each plot, the dashed orange line is the biased value $\alpha$, the solid orange line is the unbiased value $A$, and the blue line is the estimate over time $t$.}
\label{fig:SGDCT-failure}
\end{figure}
We take the parameters
\[
    \alpha = 1,\quad \epl = 0.1,\quad \sigma=0.5,\quad T=5\times10^4,\quad \Delta t = 10^{-3},
\]
and we notice that in this case, due to equation \eqref{eq:K1D}, the drift coefficient $A$ of the homogenized equation reads
\begin{equation} \label{eq:A_exp1}
A = \mathcal K \alpha = \frac{(2\pi)^2}{\left( \int_0^{2\pi} e^{-2\sin(y)} \dd y \right) \left( \int_0^{2\pi} e^{2\sin(y)} \dd y \right)} = \frac1{I_0(2)^2} \simeq 0.19,
\end{equation}
where $I$ stands for the modified Bessel function of the first kind. 

The resulting estimate over time for a single sample path is shown in \cref{fig:SGDCT-failure}. We see that the estimator converges to the biased value $\alpha$ instead of the true parameter $A$.

We also compute estimates via the filtered SGDCT scheme \eqref{eq:SGDCT_Euler_Maruyama_scheme}, which in this case is
\[
    \widehat A_{n+1}^\epl - \widehat A_n^\epl = -\eta_n Z_n^\epl X_n^\epl \widehat A_n^\epl \Delta t - \eta_n Z_n^\epl (X_{n+1}^\epl - X_n^\epl),\quad \widehat A_0^\epl = 0,
\]
where $Z_n^\epl$ is computed with the exponential filter \eqref{eq:discrete_exponential_filter}. Doing so for the same sample path as before, we obtain the second plot in \cref{fig:SGDCT-failure} and see that the estimates now converge to the right value $A$. 

We now take the parameters
\[
    \alpha = 1,\quad \varepsilon = 0.025,\quad \sigma = 0.5,\quad T = 10^3,\quad \Delta t = 1.5625\times 10^{-5},
\]
and we take $\gamma = \beta = 1$ in the learning rate.
In \cref{fig:SGDCT-delta}, we demonstrate the convergence for large $T$ of the estimate $\widehat A_T^\epl$ when the filter width $\delta$ is dependent on $\epl$. On the horizontal axis, we plot the exponent $\xi$ of $\delta = \epl^\xi$. We see that the estimator is robust as long as the filter width $\delta$ is sufficiently large, meaning that $\widehat A_T^\epl$ provides a reliable approximation of $A$ for any value of $\xi$ approximately in $[0,1]$. However, for larger $\xi$, the estimates $\widehat A_T^\epl$ move away from the correct value $A$ and finally tend to the drift coefficient $\alpha$ of the multiscale dynamics. This is in line with \cref{thm:main}, which gives a sharp transition in $\xi = 2$. We remark this exact behavior cannot be observed in practice due to the finiteness of the final time $T$ and, in particular, of the parameter $\epl$. Nevertheless, in the numerical example we demonstrate that values of $\xi$ closer to $0$, i.e., $\delta$ close to $1$, give better and more robust approximations of the unknown parameters

Now in \cref{fig:SGDCT-convergence}, we consider $\E[|\widehat{A}_T^\epl - A|^2]^{1/2}$, the $L^2$ error of the estimator for a fixed value of $\epl$, as $T\to\infty$ and show that it follows the expected $\mathcal O(T^{-1/2})$ rate highlighted in \cref{rem:rate}. We take the parameters
\[
    \alpha = 1,\quad \epl = 0.1,\quad \sigma=0.5,\quad T=10^4,\quad \Delta t = 10^{-3}.
\]
Because we cannot compute the expected value exactly, we instead compute an approximate error using the Monte Carlo estimate
\[
    \text{Approximate $L^2$ Error at time $n$} = \left(\frac1M \sum_{i=1}^{M} (\widehat{A}_n^{\epl,i} - A)^2\right)^{1/2},
\]
where $M=50$ is the number of multiscale data paths which we sample, and $\widehat{A}_n^{\epl,i}$ is the SGDCT value at time $n$ corresponding to the $i$-th sample. We indeed observe that the approximate $L^2$ error has the same slope as the $\mathcal O(T^{-1/2})$ reference line.

\begin{figure}
    \centering
    \includegraphics[scale=0.5]{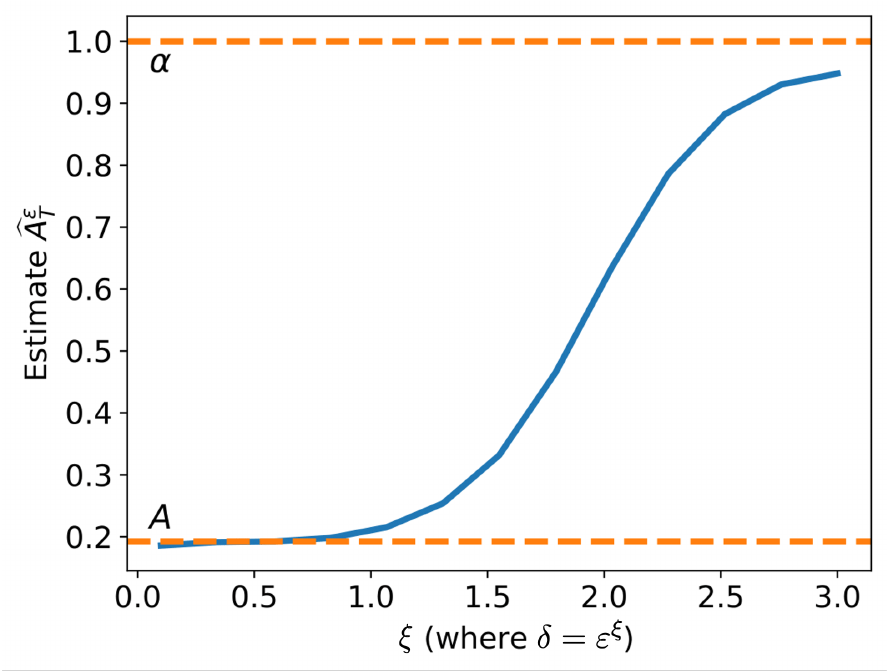}
    \caption{Effect of filter width $\delta$ dependence on $\varepsilon$, for the example in \cref{sec:verification}. The filter width is given by $\delta = \varepsilon^\xi$, where $\xi \in (0,3)$ is the variable on the horizontal axis. The top dashed orange line corresponds to the biased value $\alpha$, while the bottom dashed orange line corresponds to the unbiased value $A$. The blue line gives the value of the exponential filter estimator at the final time $T$.}
    \label{fig:SGDCT-delta}
\end{figure}

\begin{figure}
    \centering
    \includegraphics[scale=0.48]{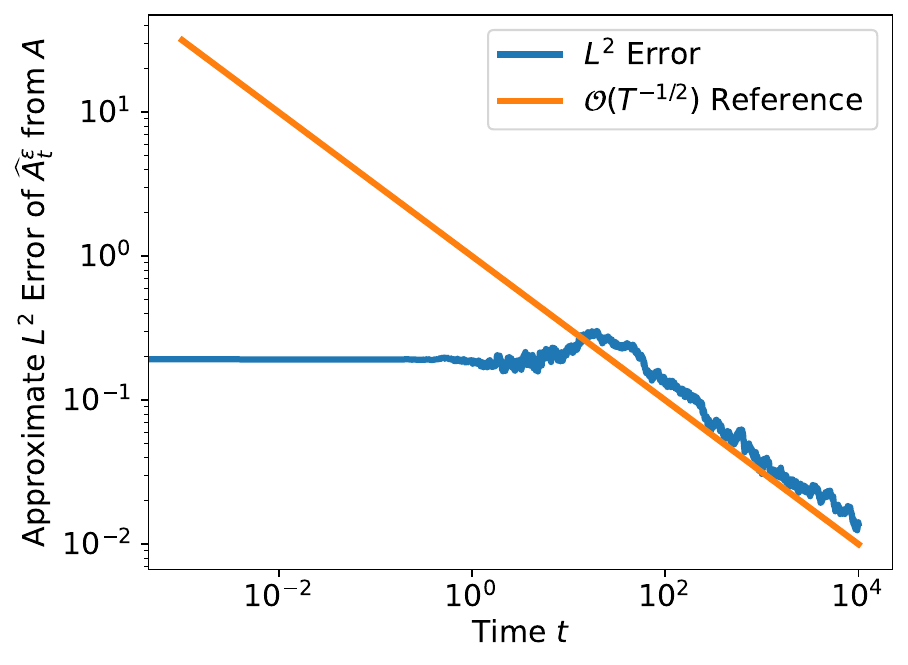}
    \caption{Rate of convergence in time of SGDCT estimates $\widehat A_T^\epl$ with exponential filtered data, for the example in \cref{sec:verification}. A Monte Carlo approximation of the $L^2$ error of the estimates $\widehat A_T^\epl$ from the true value $A$ is given in blue, and a reference line giving an $\mathcal{O}(T^{-1/2})$ rate is given in orange.}
    \label{fig:SGDCT-convergence}
\end{figure}

\subsection{Multidimensional drift coefficient} \label{sec:multidimensional}
In this section, we now consider
\[
    V(x;\alpha) = \alpha_1 \frac{x^4}{4} - \alpha_2 \frac{x^2}{2} \qquad \text{and} \qquad p(x, y) = \sin(y),
\]
so that
\[
    (\mathcal V^\epl)'(x; \alpha) = \alpha_1 x^3 - \alpha_2 x + \frac{1}{\epl}\cos\left(\frac{x}{\epl}\right).
\]
We then estimate $A = \begin{bmatrix} a_1 & a_2 \end{bmatrix}^\top$ in the expansion of $\widetilde b(x; A) = A^\top U(x)$, where $U(x) = \begin{bmatrix} x^3 & -x\end{bmatrix}^\top$.
This estimation is done via the filtered SGDCT scheme \eqref{eq:SGDCT_Euler_Maruyama_scheme}, which in this case is
\[
\widehat A_{n+1}^\epl - \widehat A_n^\epl = -\eta_n \begin{bmatrix} (Z_n^\epl)^3(X_n^\epl)^3 & -(Z_n^\epl)^3X_n^\epl\\ -Z_n^\epl(X_n^\epl)^3 & Z_n^\epl X_n^\epl\end{bmatrix}\widehat A_n^\epl \Delta t - \eta_n \begin{bmatrix}(Z_n^\epl)^3\\ -Z_n^\epl\end{bmatrix}(X_{n+1}^\epl - X_n^\epl),\quad \widehat A_0^\epl = 0,
\]
where $Z_n^\epl$ is computed with the exponential filter \eqref{eq:discrete_exponential_filter} or the moving average \eqref{eq:discrete_moving_average} with $\delta=1$. We use the parameters
\[
    \alpha_1 = 1,\quad \alpha_2 = 2,\quad \epl = 0.1,\quad \sigma=0.5,\quad T=10^5,\quad \Delta t = 10^{-3},
\]
which, repeating the computation in \eqref{eq:A_exp1}, give the homogenized drift coefficients 
\[
    A_1 = \frac1{I_0(2)^2} \simeq 0.19, \qquad A_2 = \frac2{I_0(2)^2} \simeq 0.38.
\]

The results are given in \cref{fig:exp-filtered-2d} and \cref{fig:ma-filtered-2d}, respectively. In each of these figures, we see the path of the SGDCT estimator corresponding to a single sample $X_n$ with the components $A_1$ and $A_2$ first in separate plots, and then in the same plot, where the color of the curve corresponds to time. In each plot (a) and (b), the orange horizontal line corresponds to the true parameter values $A_1$ and $A_2$. In the plot (c), the star represents the true parameter $A$. In both figures, we see the sample estimates of $A_1$ and $A_2$ converging to the true parameters as time increases. We note that a larger time $T$ was required to observe convergence to $A_1$ and $A_2$. In general, we observed that a larger time $T$ was required to observe convergence when there are more parameters to estimate.

\begin{figure}
    \centering
    \begin{tabular}{cc}
    \includegraphics[scale=0.48]{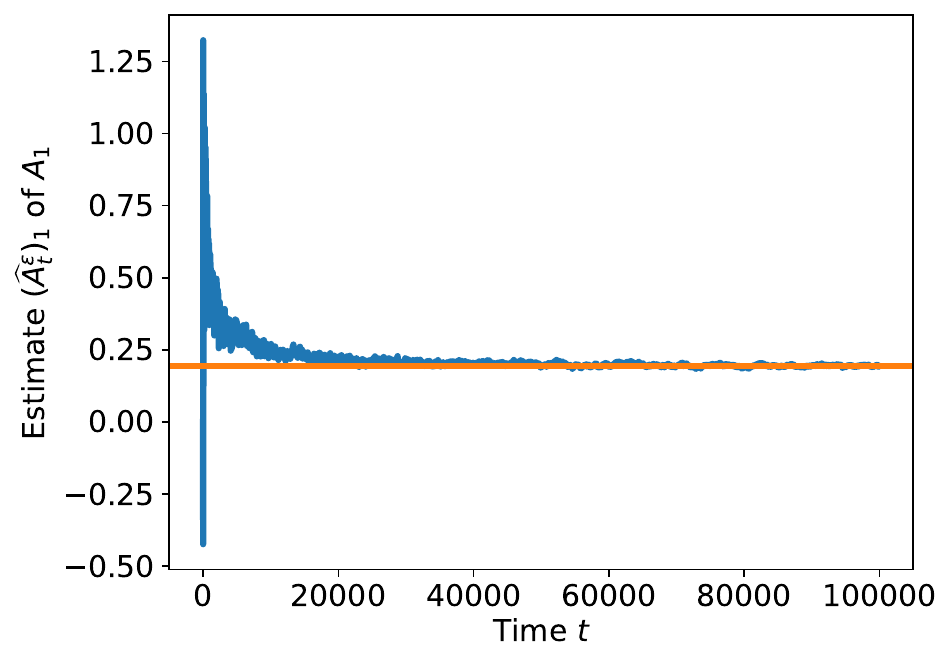} & \includegraphics[scale=0.48]{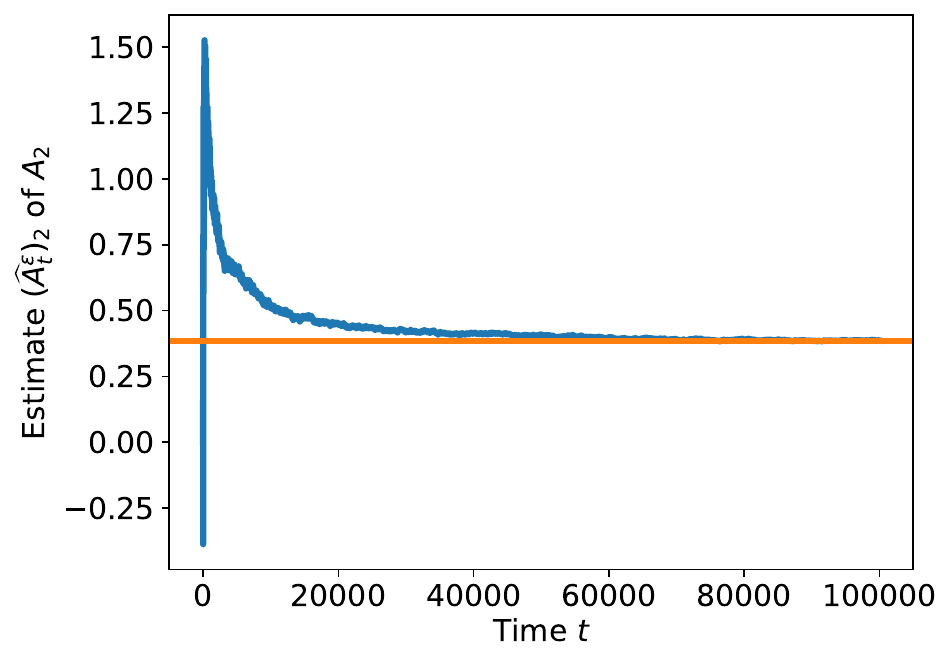}\\
    (a) Sample of SGDCT for $A_1$. & (b) Sample of SGDCT for $A_2$.
    \end{tabular}
    \begin{tabular}{c}
    \includegraphics[scale=0.48]{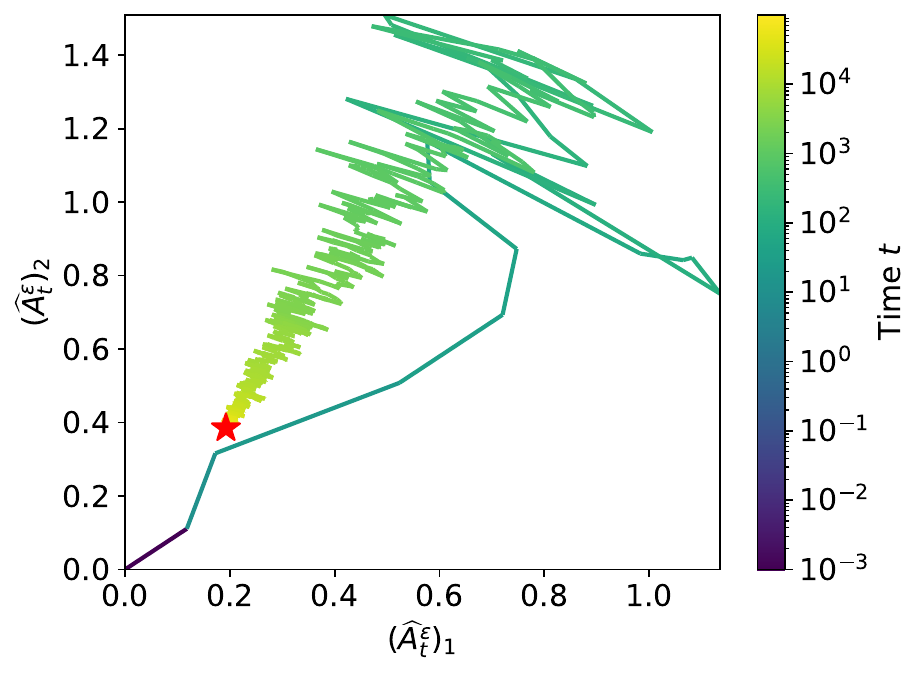}\\
    (c) Sample of SGDCT for $A_1$ and $A_2$.
    \end{tabular}
    \caption{Sample estimates using exponential filtered SGDCT, for the multidimensional drift coefficient in \cref{sec:multidimensional}. In (a) and (b), the blue line is the estimate over time $t$ for $A_1$ and $A_2$, respectively. The orange lines correspond to $A_1$ and $A_2$, respectively. In (c), the curve is the path of the estimate of $(A_1, A_2)$ with the color of the curve representing the time of the estimate and the red star representing to true value $(A_1, A_2)$.}
    \label{fig:exp-filtered-2d}
\end{figure}

\begin{figure}
    \centering
    \begin{tabular}{cc}
    \includegraphics[scale=0.48]{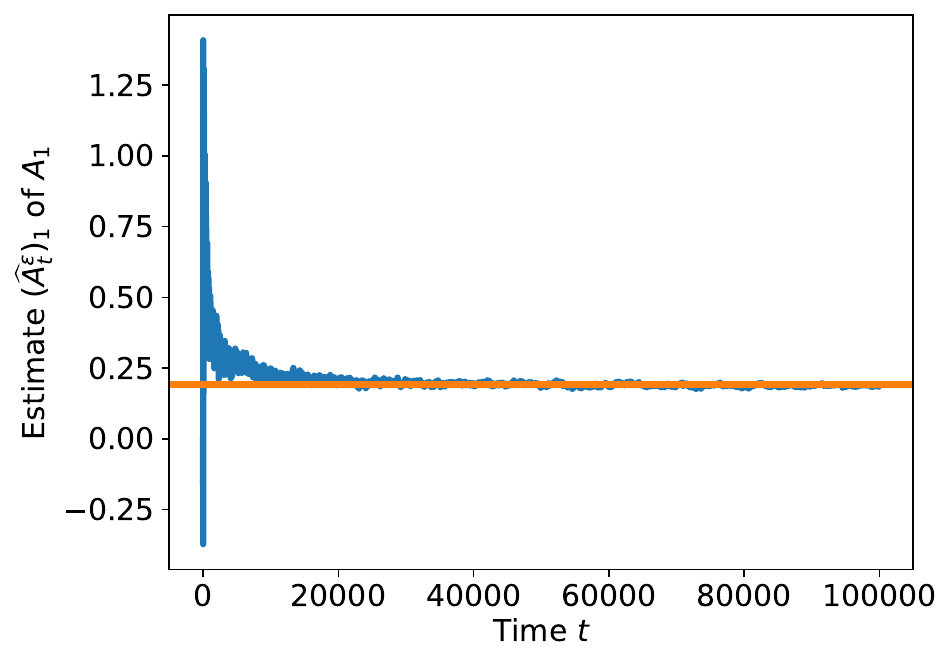} & \includegraphics[scale=0.48]{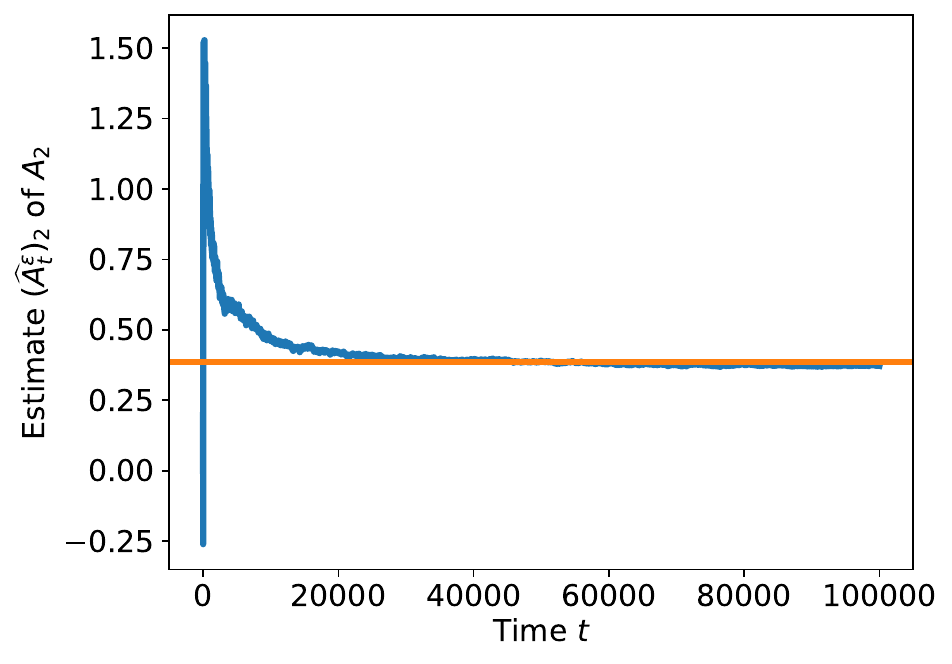}\\
    (a) Sample of SGDCT for $A_1$. & (b) Sample of SGDCT for $A_2$.
    \end{tabular}
    \begin{tabular}{c}
    \includegraphics[scale=0.48]{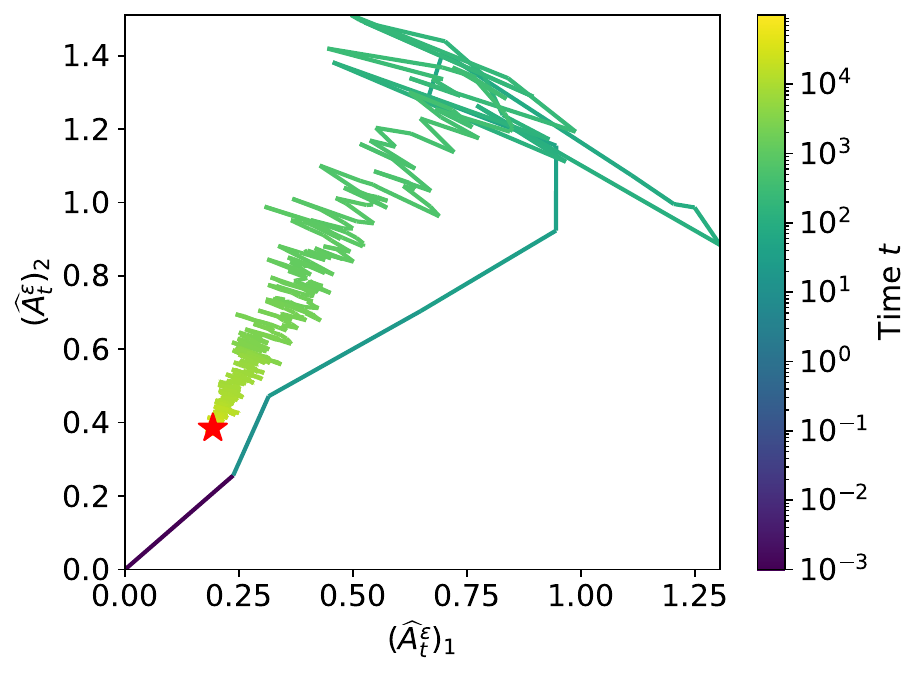}\\
    (c) Sample of SGDCT for $A_1$ and $A_2$.
    \end{tabular}
    \caption{Sample estimates using the moving average SGDCT, for the multidimensional drift coefficient in \cref{sec:multidimensional}. In (a) and (b), the blue line is the estimate over time $t$ for $A_1$ and $A_2$, respectively. The orange lines correspond to $A_1$ and $A_2$, respectively. In (c), the curve is the path of the estimate of $(A_1, A_2)$ with the color of the curve representing the time of the estimate and the red star representing to true value $(A_1, A_2)$.}
    \label{fig:ma-filtered-2d}
\end{figure}

\subsection{Learning the drift function} \label{sec:drift}
In this example, we estimate a more generic drift function $b$ given by a nonseparable multiscale potential. We have
\[
    V(x) = \frac{x^4}{4} - \frac{x^2}{2} \qquad \text{and} \qquad p(x, y) = \frac{x^2}{2}\cos(y)\mathbbm{1}\{|x| \le 2\},
\]
so that
\[
    (\mathcal{V}^\epl)'(x) = x^3 - x + x\cos\left( \frac{x}\epl \right) \mathbbm{1}\{|x|\le 2\} - \frac{x^2}{2\epl}\sin\left( \frac{x}\epl \right)\mathbbm{1}\{|x|\le 2\}.
\]
We remark that in this setting the homogenized drift term $b$ can be computed analytically. In fact, due to \eqref{eq:K1D}, we have
\begin{equation}
\mathcal K(x) = \frac{(2\pi)^2}{\left( \int_0^{2\pi} e^{- \frac1\sigma \frac{x^2}{2}\cos(y)\mathbbm{1}\{|x| \le 2\}} \dd y \right)\left( \int_0^{2\pi} e^{\frac1\sigma \frac{x^2}{2}\cos(y)\mathbbm{1}\{|x| \le 2\}} \dd y \right)} = I_0 \left( \frac{x^2}{2\sigma} \right)^{-2} \mathbbm{1}\{|x| \le 2\} + \mathbbm{1}\{|x| > 2\},
\end{equation}
which, by equation \eqref{eq:drift_hom}, implies
\begin{equation} \label{eq:drift_hom_exp3}
b(x) = \begin{cases}
V'(x) & \text{if } \abs{x} > 2, \\
V'(x) I_0 \left( \frac{x^2}{2\sigma} \right)^{-2} + x I_1 \left( \frac{x^2}{2\sigma} \right) I_0 \left( \frac{x^2}{2\sigma} \right)^{-3} & \text{if } \abs{x} \le 2.
\end{cases}
\end{equation}
We then aim to estimate $A = \begin{bmatrix} a_1 & a_2 & a_3 & a_4 \end{bmatrix}^\top$ in the expansion of $\widetilde{b}(x; A) = A^\top U(x)$ with
\[
    U(x) = \begin{bmatrix} 1 & x & x^2 & x^3\end{bmatrix}^\top.
\]
We do this via the filtered SGDCT scheme \eqref{eq:SGDCT_Euler_Maruyama_scheme} with the above choice of $U$ and the exponential filter with $\delta=1$. We use the parameters
\[
    \epl \in \{0.05, 0.1\},\quad \sigma = 2,\quad T\in\{10, 10^2, 10^3\},\quad \Delta t = 1.25\times10^{-4},
\]
and we set $\gamma = 5/2$ and $\beta = 10$ in the learning rate. The plot of the estimated expansion $\widetilde b$ of $b$ given in \eqref{eq:drift_hom_exp3} for a single sample from \eqref{eq:SDE_Euler_Maruyama_scheme} compared to the true value of $b$ is shown in \cref{fig:experiment3-varying-time}. We see that the approximation $\widetilde b$ matches the shape of $b$, with better approximations for larger $T$. Moreover, even if the difference is not large, the approximation is also better for $\varepsilon = 0.05$ than for $\varepsilon=0.1$.

We now consider the parameters
\[
    \epl \in \{0.05, 0.1, 0.2\},\quad \sigma \in \{1, 2, 4\},\quad T = 10^5,\quad \Delta t = 1.25\times10^{-4},
\]
with the same learning rate, i.e., $\gamma = 5/2$ and $\beta=10$, and the same slow-scale potential $V(x) = \frac{x^4}{4} - \frac{x^2}{2}$. We aim to estimate $A = \begin{bmatrix} a_1 & a_2 & a_3 & a_4 \end{bmatrix}^\top$ with $U(x) = \begin{bmatrix} 1 & x & x^2 & x^3\end{bmatrix}^\top$. We now use the moving average with $\delta=1$. The resulting estimates $\widetilde b(x)$ of the true drift function $b(x)$ are shown in \cref{fig:experiment3-two-term-potential}. We observe that, as expected, results get worse for larger values of $\epl$, and, in particular, $\epl$ must be sufficiently small in order for the approximations to be reliable. Moreover, we see that increasing $\sigma$ smooths out the true drift function, so that it can be better approximated by a cubic function, and as such, the estimates improve for larger $\sigma$.

\begin{figure}
    \centering
    \begin{tabular}{cc}
    \includegraphics[scale=0.5]{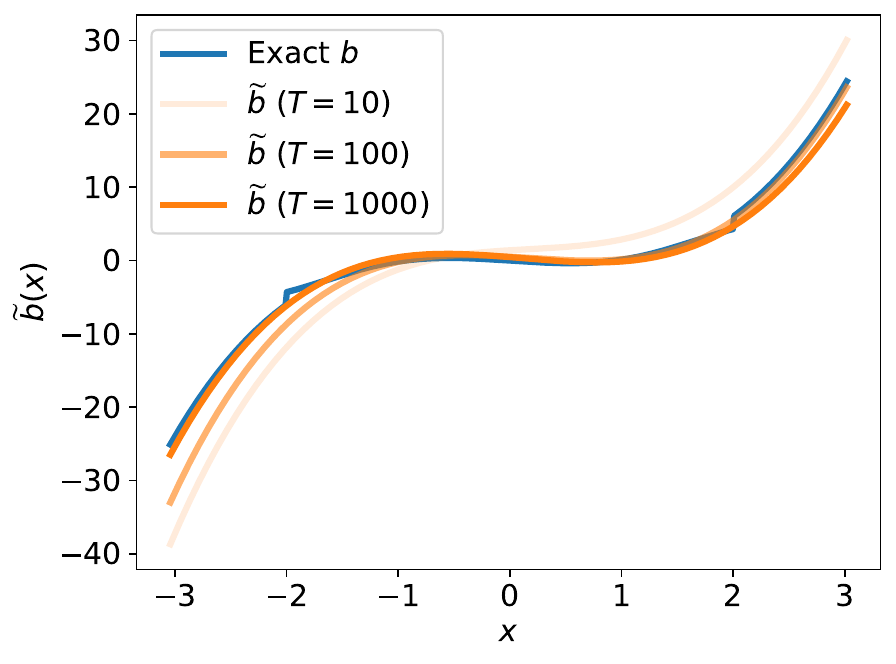} & \includegraphics[scale=0.5]{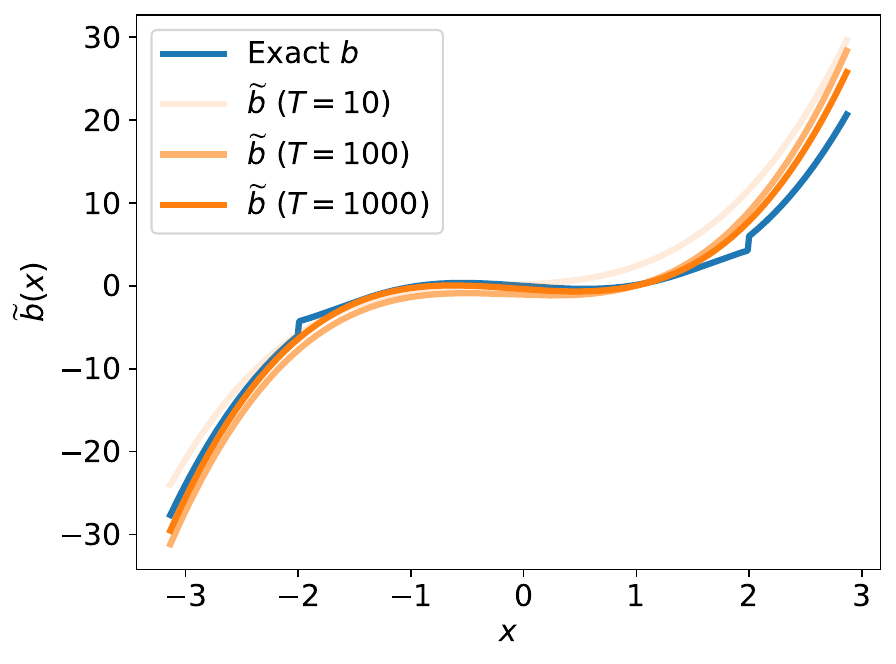}\\
    (a) $\varepsilon = 0.05$ & (b) $\varepsilon = 0.1$
    \end{tabular}
    \caption{Learned drift function $\widetilde b(x)$ at $T=10, 100, 1000$ for SGDCT with exponential filter, for different values of $\epl=0.05, 0.1$, for the example in \cref{sec:drift}. The exact drift function $b$ is given in blue, while the approximations corresponding to $\widetilde b$ are given in shades of orange. Darker shades of orange correspond to the estimates for larger times.}
    \label{fig:experiment3-varying-time}
\end{figure}

\begin{figure}
    \centering
    \begin{tikzpicture}
        \draw (-5.25, 6.5) node {$\sigma = 1$};
        \draw (-0.8, 6.5) node {$\sigma = 2$};
        \draw (3.6, 6.5) node {$\sigma = 4$};
        \draw (-9.25, 4.25) node {$\varepsilon = 0.05$};
        \draw (-9.25, 0.5) node {$\varepsilon = 0.1$};
        \draw (-9.25, -3.5) node {$\varepsilon = 0.2$};
        \draw (-1.5, 0) node[inner sep=0] {
            \centering
            \begin{tabular}{c}
            \includegraphics[scale=0.85]{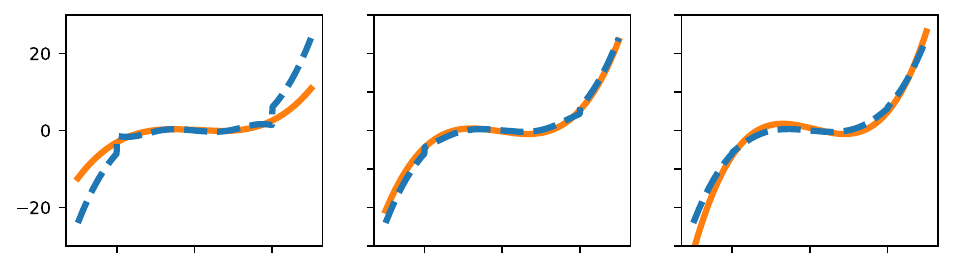}\\
            \includegraphics[scale=0.85]{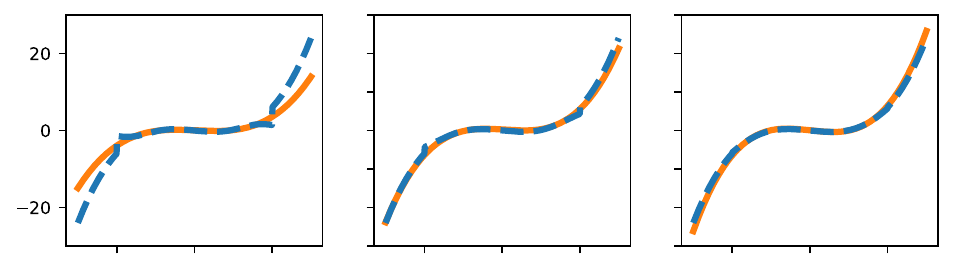}\\
            \includegraphics[scale=0.85]{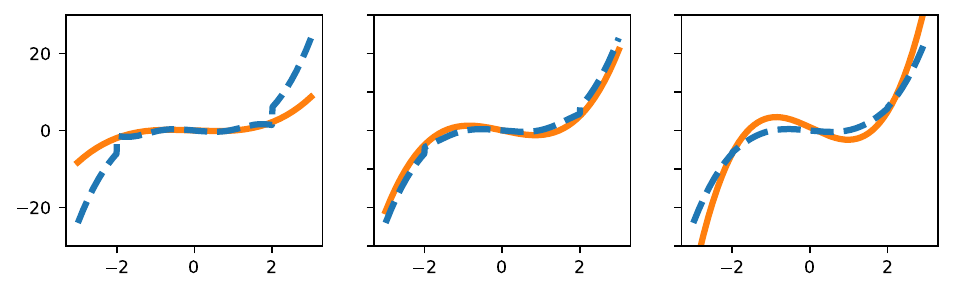}
            \end{tabular}
        };
    \end{tikzpicture}
    \caption{Learned drift functions $\widetilde b(x)$ (solid orange) compared to true drift function $b(x)$ (dashed blue), varying the multiscale parameter $\epl$ and the diffusion coefficient $\sigma$, for the example in \cref{sec:drift}. Estimates are obtained using data filtered with a moving average.}
    \label{fig:experiment3-two-term-potential}
\end{figure}
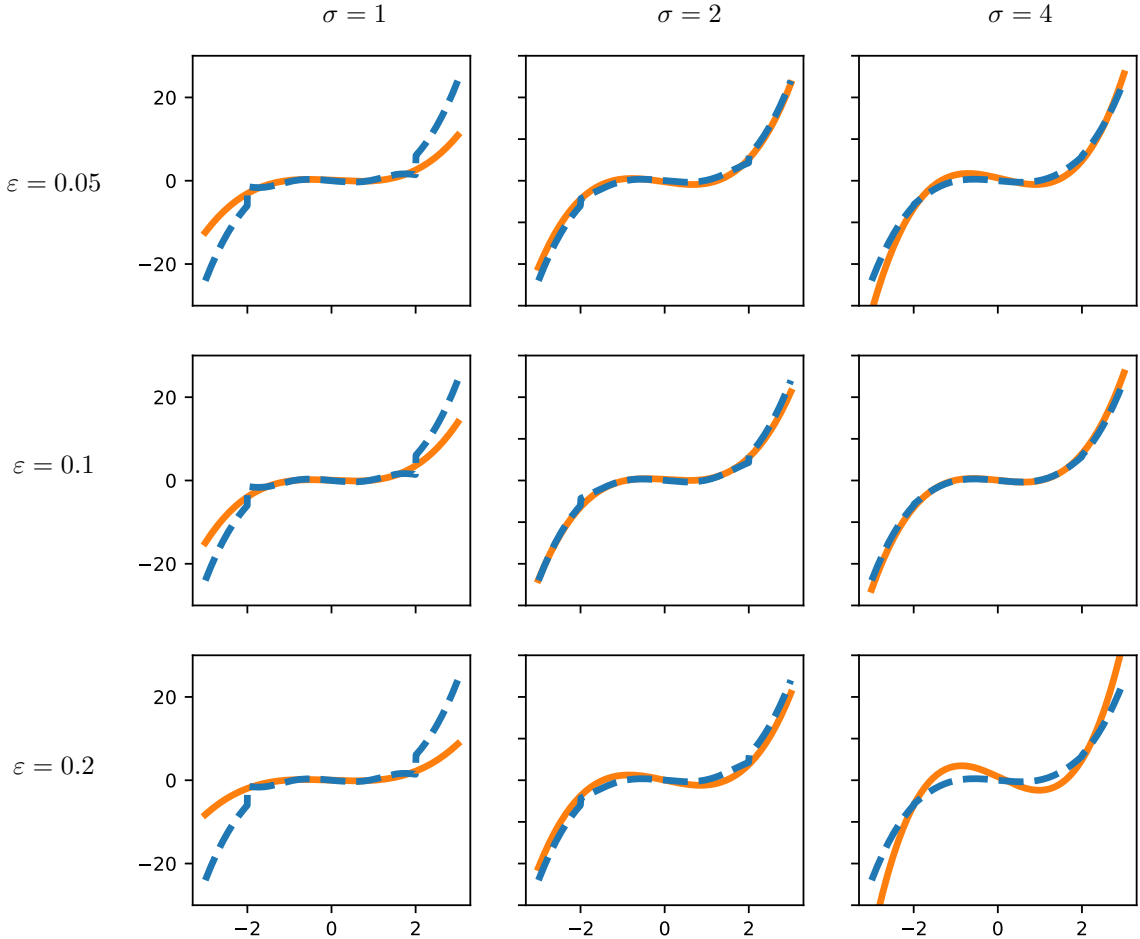

We finally consider estimating the slow-scale potential with $3$ and $5$ basis functions to investigate the effect of having too few or too many basis functions compared to the $4$ basis functions needed to represent the slow-scale potential. We consider the parameters
\[
    \varepsilon = 0.1,\quad \sigma = 2,\quad T = 10^3,\quad \Delta t = 1.25 \times 10^{-4},
\]
with the same learning rate $\gamma = 5/2$ and $\beta=10$ and the same slow-scale potential $V(x) = \frac{x^4}{4}-\frac{x^2}{2}$. We aim to estimate $A = \begin{bmatrix} a_1 & a_2 & a_3 \end{bmatrix}^\top$ with $U(x) = \begin{bmatrix} 1 & x & x^2\end{bmatrix}^\top$, $A = \begin{bmatrix} a_1 & a_2 & a_3 & a_4\end{bmatrix}^\top$ with $U(x) = \begin{bmatrix} 1 & x & x^2 & x^3\end{bmatrix}^\top$, and $A = \begin{bmatrix} a_1 & a_2 & a_3 
 & a_4 & a_5\end{bmatrix}^\top$ with $U(x) = \begin{bmatrix} 1 & x & x^2 & x^3 & x^4\end{bmatrix}^\top$. The results are shown in Figure \ref{fig:experiment3-varying-basis}. As we see from Figure \ref{fig:experiment3-varying-basis}(a), we are not able to represent the drift function with too few basis functions. When there is a larger number of basis functions than is needed to approximate the drift function as in Figure \ref{fig:experiment3-varying-basis}(c), we are still able to approximate the drift function, but a larger time $T$ is required for convergence of the estimator compared to the time required for a smaller number of basis functions in our approximation, like the plot in Figure \ref{fig:experiment3-varying-basis}(b), making the results slightly worse.

\begin{figure}
    \centering
    \begin{tabular}{ccc}
    \includegraphics[scale=0.34]{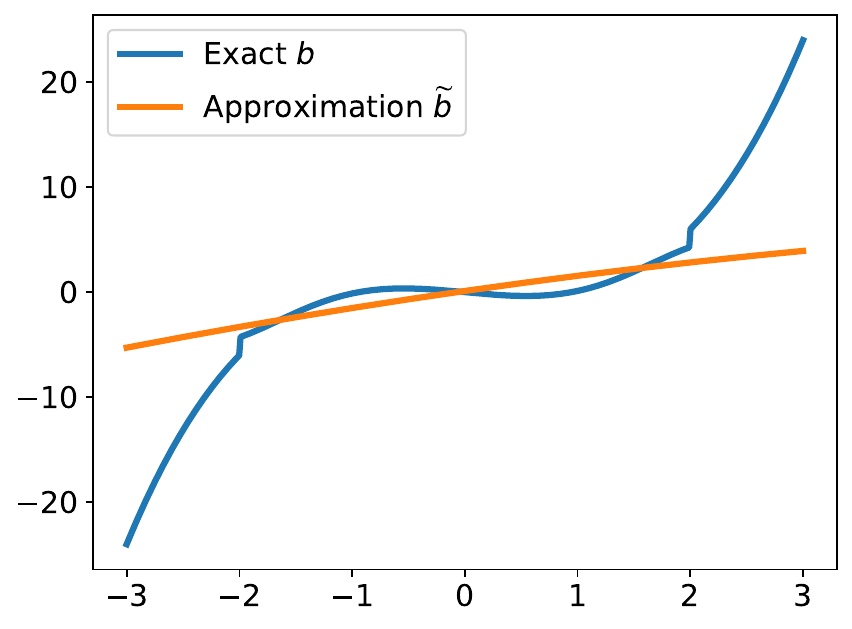} & 
    \includegraphics[scale=0.34]{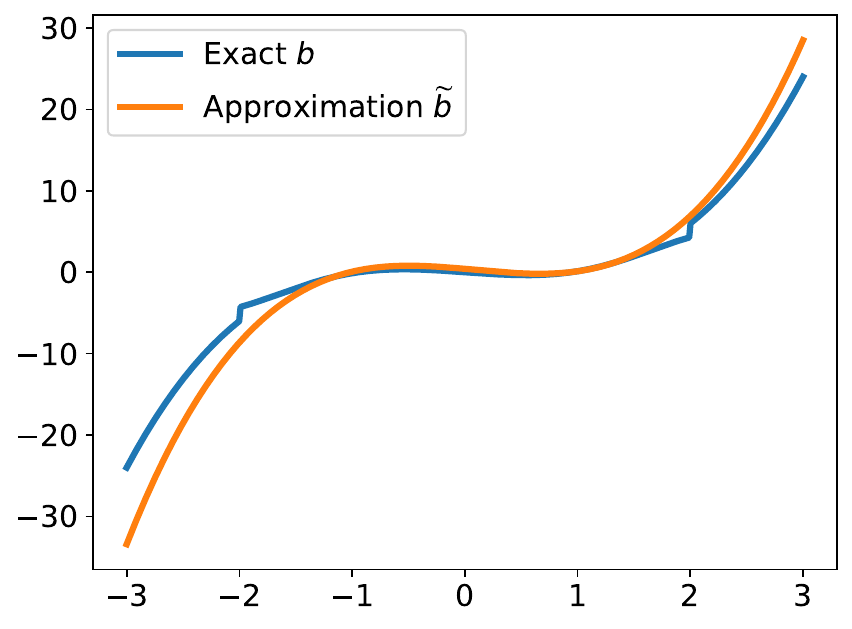} &
    \includegraphics[scale=0.34]{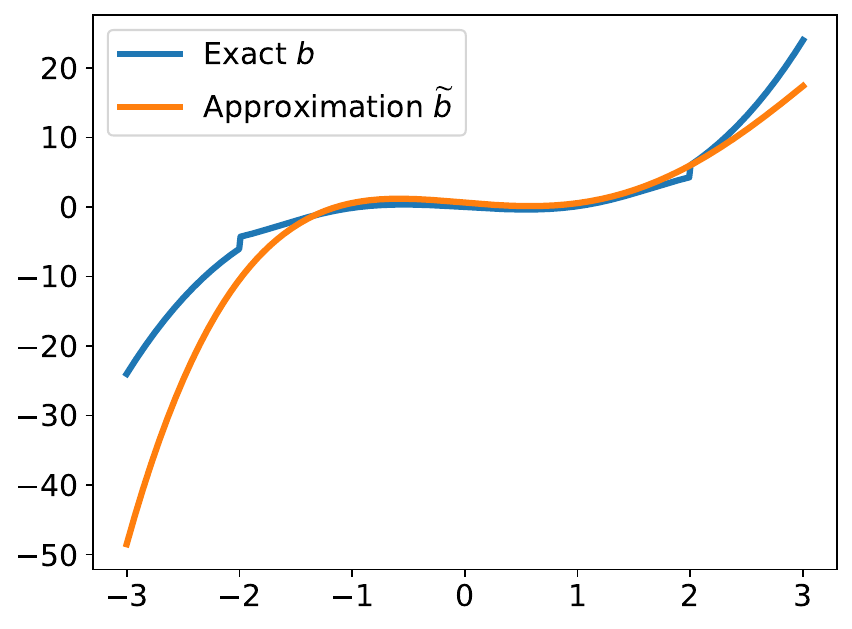}\\
    (a) 3 basis functions & (b) 4 basis functions & (b) 5 basis functions
    \end{tabular}
    \caption{Learned drift functions $\widetilde b(x)$ at time $T$ (in orange) compared to true drift function $b(x)$ (in blue) for (a) 3 basis functions, (b) 4 basis functions, and (c) 5 basis functions in the approximation $\widetilde b(x)$, for the example in \cref{sec:drift}. Estimates are obtained using data filtered with a moving average.}
    \label{fig:experiment3-varying-basis}
\end{figure}

\section{Conclusion} \label{sec:conclusion}

In this work we focused on multiscale diffusion processes of the Langevin type, considering also an example of nonseparable multiscale potentials, meaning that the fast-scale oscillations are dependent on the slow-scale component. The theory of the homogenization gives the existence of a homogenized SDE which only captures the behavior of the slowest scale. We aimed to infer the drift term of this effective equation given continuous-time observations from the multiscale dynamics. We considered a semiparametric approach and wrote the drift function in terms of a truncated series of basis functions. We then proposed to estimate the coefficients of this expansion coupling SGDCT and filtered data in order to overcome the issue of incompatibility between data and model. We derived a rigorous theoretical analysis in the particular setting where only a one-dimensional drift coefficient is unknown, the confining potential of the Langevin dynamics is separable, and the state space is compact. In this context, we proved the asymptotic unbiasedness of our novel estimator in the limit of infinite time and when the multiscale parameter vanishes. Moreover, we performed a series of numerical experiments which showed that our estimator can be successfully employed in more general frameworks. Therefore, it would be interesting to extend the range of applicability of the theoretical analysis to more general scenarios by lifting some of the assumptions. This would first require preliminary studies related to the solution of the Poisson problem for hypoelliptic operators in unbounded domains, function approximation, and stochastic delay differential equations. Finally, another interesting future direction of research is the study of the rate of convergence and possibly a central limit theorem for the SGDCT estimator with filtered data. This would indeed give more precise information about the accuracy of the estimator in the finite regime, and indications on the choice of the parameters in the learning rate. We remark that the rate of convergence with respect to the multiscale parameter is dependent on the rate of weakly convergence of the invariant measure of the multiscale diffusion process together with its filtered data towards the homogenized invariant measure. Hence, we first plan to analyze the convergence in law of the joint process of the original and filtered data towards its homogenized counterpart. 

\vspace{2ex}
\noindent\textit{Data availability statement.} The code used to produce the numerical experiments in this article can be found on the Github repository \url{https://github.com/maxhirsch/Multiscale-SGDCT} \cite{PaperCode}.

\vspace{2ex}
\noindent\textit{Acknowledgements.} We thank the anonymous reviewers whose comments and suggestions helped improve and clarify this manuscript. AZ is grateful to Assyr Abdulle for introducing him to the world of research. This material is based upon work supported by the National Science Foundation Graduate Research Fellowship Program under Grant No.\ DGE 2146752. Any opinions, findings, and conclusions or recommendations expressed in this material are those of the authors and do not necessarily reflect the views of the National Science Foundation. MH also acknowledges the ThinkSwiss Research Scholarship and the EPFL Excellence Research Scholarship. AZ was partially supported by the Swiss National Science Foundation, under grant No. 200020\_172710.

\bibliographystyle{abbrv}
\bibliography{bibliography}

\end{document}